\documentclass[a4paper]{article}
\usepackage{blindtext}
\usepackage[utf8]{inputenc}
\usepackage[margin=1in]{geometry}
\usepackage{amsmath}
\usepackage{amssymb}
\usepackage{amsthm}
\usepackage{amscd}
\usepackage[affil-it]{authblk}

\newcommand{\ZZ}{\Bbb Z}
\newcommand{\RR}{\Bbb R}
\newcommand{\TT}{\Bbb T}
\newcommand{\NN}{\Bbb N}
\newcommand{\QQ}{\Bbb Q}
\newcommand{\CC}{\Bbb C}
\newcommand{\ip}[1]{\langle #1 \rangle}

\newcommand{\widetidle}{\widetilde}

\newcommand{\varpesilon}{\varepsilon}

\newcommand{\varespilon}{\varepsilon}
\newcommand{\vaerpsilon}{\varepsilon}

\newcommand{\actson}{\curvearrowright}
\newcommand{\acston}{\curvearrowright}

\newtheorem{question}{Question}

\newtheorem{?}{Question}
\newtheorem{theorem}{Theorem}
\newtheorem{definition}[theorem]{Definition}
\newtheorem{proposition}[theorem]{Proposition}
\newtheorem{cor}[theorem]{Corollary}
\newtheorem{lemma}[theorem]{Lemma}

\DeclareMathOperator{\IE}{IE}

\DeclareMathOperator{\Span}{Span}

\DeclareMathOperator{\Map}{Map}

\DeclareMathOperator{\supp}{supp}

\DeclareMathOperator{\Det}{det}

\DeclareMathOperator{\Hom}{Hom}

\DeclareMathOperator{\Tr}{Tr}

\numberwithin{theorem}{section}

\begin{document}
\begin{titlepage}
\title{ Independence Tuples and Deninger's Problem }        
\author{Ben Hayes}
\affil{\emph{Vanderbilt University}\\
         \emph{Stevenson Center}\\
         \emph{Nashville, TN 37240}\\
         \emph{benjamin.r.hayes@vanderbilt.edu}}
\date{\today}
\maketitle

\begin{abstract} We define  modified versions of the independence tuples for sofic entropy developed in \cite{KerrLi2}. Our first modification uses an $\ell^{q}$-distance instead of an $\ell^{\infty}$-distance. It turns out this produces the same version of independence tuples (but for nontrivial reasons), and this allows one added flexibility. Our second modification considers the ``action" a sofic approximation gives on $\{1,\dots,d_{i}\},$ and forces our independence sets $J_{i}\subseteq\{1,\dots,d_{i}\}$ to be such that $\chi_{J_{i}}-u_{d_{i}}(J_{i})$ (i.e. the projection of $\chi_{J_{i}}$ onto mean zero functions) spans a representation of $\Gamma$ weakly contained in the left regular representation. This modification is motivated by the results in \cite{Me6}. Using both of these modified versions of independence tuples we prove that if $\Gamma$ is sofic, and $f\in M_{n}(\ZZ(\Gamma))\cap GL_{n}(L(\Gamma))$ is not invertible in $M_{n}(\ZZ(\Gamma)),$ then $\Det_{L(\Gamma)}(f)>1.$ This extends a consequence of the work in \cite{Me5} and \cite{KerrLi2} where one needed $f\in M_{n}(\ZZ(\Gamma))\cap GL_{n}(\ell^{1}(\Gamma)).$ As a consequence of our work, we show that if $f\in M_{n}(\ZZ(\Gamma))\cap GL_{n}(L(\Gamma))$ is not invertible in $M_{n}(\ZZ(\Gamma))$ then $\Gamma\actson (\ZZ(\Gamma)^{\oplus n}/\ZZ(\Gamma)^{\oplus n}f)^{\widehat{}}$ has completely positive topological entropy with respect to any sofic approximation.\end{abstract}


\tableofcontents

\end{titlepage}

\section{Introduction}

	This paper is concerned with a modification of  independence tuples in the case of sofic topological entropy due to Kerr-Li in \cite{KerrLi2}. We remark that the definition of sofic topological entropy is due to Kerr-Li in \cite{KLi}, following on the seminal work of Bowen on sofic measure-theoretic entropy  in \cite{Bow}. Independence tuples were first developed in \cite{KerrLiIndepened} for actions of amenable groups. Positivity of topological entropy is equivalent to having a nondiagonal independence pair, and this can be viewed as a topological version of the fact that a measure-preserving action of an amenable group must have a weakly mixing factor. Using independence tuples, Kerr-Li showed that if a topological action has positive entropy, then the action must exhibit some chaotic behavior (see e.g. \cite{KerrLi2} Theorem 8.1). Let us briefly mention the combinatorial version of independence. We say that a tuple $(A_{i,1},\dots,A_{i,k})_{i\in J}$ of subsets of a set $A$ are independent, if  for every $c\colon J\to \{1,\dots,k\}$ we have
\[\bigcap_{i\in J}A_{i,c(i)}\ne \varnothing.\]
The name coming from the case when $(A_{i,1},\dots,A_{i,k})_{i\in J}$ are (probabilistically) independent partitions in a probability space. If $\Gamma$ is a countable discrete group acting on a set $X,$ and $(A_{1},\dots,A_{k})$ are subsets of $X,$ we  call a finite $J\subseteq \Gamma$ an independence set for $(A_{1},\dots,A_{k})$ if $\{(s^{-1}A_{1},\dots,s^{-1}A_{k})\}_{s\in J}$ is independent. When $\Gamma$ is amenable and $F\subseteq\Gamma$ is finite, we let  $\phi_{A}(F)$ is the maximal cardinality of a subset of $F$ which is an independent set for $A.$ We can then define the independence density of $A=(A_{1},\dots,A_{k}),$ denoted $I(A),$  to be the limit  of $\frac{\phi_{A}(F_{n})}{|F_{n}|}$ where  $F_{n}$ is a F\o lner sequence. In the case $X$ is compact and the action is by homeomorphisms, we say that a tuple $x=(x_{1},\dots,x_{k})$ is an independence tuple if every tuple $U=(U_{1},\dots,U_{k})$ where $U_{j}$ is a neighborhood of $x_{j}$ we have $I(U)>0.$ This definition is due to Kerr-Li in \cite{KerrLiIndepened}. 

	To generalize to the case of sofic groups (defined in the next section), Kerr-Li considered a sofic approximation (again we define this in the next section)
\[\sigma_{i}\colon\Gamma\to S_{d_{i}}\]
and abstracted the internal independent subsets of $\Gamma$ considered in the amenable case to  \emph{external} independent subsets of $\{1,\dots,d_{i}\}$ in the sofic case. In this manner they defined what it means for a tuple $(x_{1},\cdots,x_{k})\in X^{k}$ to be an independence tuple for the action $\Gamma\actson X$ with respect to some fixed sofic approximation $\sigma_{i}\colon\Gamma\to S_{d_{i}}.$ Moreover, they showed that $\Gamma\actson X$ has positive entropy with respect to $\sigma_{i}\colon\Gamma\to S_{d_{i}}$ if and only if there is a nondiagonal  independence pair in $X^{2}.$ This can be viewed as a topological version of the fact that a probability measure-preserving action with positive entropy must have a weakly mixing factor.

We give two alternate  versions of an independence tuple for actions of sofic groups. For the first it is useful to rephrase the definition in terms of metrics. Let $\rho$ be a compatible metric on $X.$ The condition
\[\bigcap_{g\in J}g^{-1}U_{c(g)}\ne\varnothing\]
can be replaced by the similar condition that there is a $x\in X$ so that
\[\max_{g\in J}\rho(gx,gx_{c(s)})<\varepsilon.\]
Equivalently, for $1\leq p\leq\infty,$ let us define $\rho_{p}$ on $X^{J}$ by
\[\rho_{p,J}(x,y)^{p}=\frac{1}{|J|}\sum_{g\in J}\rho(x(g),y(g))^{p}\mbox{ if $p<\infty$}\]
\[\rho_{\infty,J}(x,y)=\sup_{j\in J}\rho(x(j),y(j)).\]
Then we are considering the condition that
\[\rho_{\infty,J}(O(x),x_{c(\cdot)})<\varepsilon\]
where $O\colon X\to X^{J}$ is defined by $O(x)(g)=gx.$  One can rephrase the sofic version of independence sets in terms of a similar $\ell^{\infty}$-product metric.  We define an a priori different version of independent set using an $\ell^{p}$-product metric. This is a priori weaker than the $\ell^{\infty}$-product version, however by an application of the Sauer-Shelah theorem we can show that they are equivalent. While it appears that we have thus accomplished nothing, this actually gives us an added degree of flexibility as the $\ell^{2}$-product metric will be more useful to us. The technique of using $\ell^{p}$-metrics instead of $\ell^{\infty}$-metrics was first used by  Li in \cite{Li2}. We believe this is a very important technique, which often gives one added flexibility needed to prove results in entropy theory. We mention that we have already exploited this in \cite{Me4},\cite{Me5},\cite{Me6}. It is quite useful when one wishes to apply Hilbert space techniques as these are phrased better in terms of the $\ell^{2}$-product metric. This is precisely the purpose of their use in \cite{Me4},\cite{Me5},\cite{Me6} and we believe this is crucial for those results, as well as the results in this paper.

 The second version of independence tuples is one in which we control the  translates of an independence set $J$ by the left regular representation (in a sense to be made more precise in Section \ref{S:WCC}), and moreover only consider partitions
\[c\colon J\to \{1,\dots,k\}\]
where each of the pieces $c^{-1}(\{l\}),1\leq l\leq k$ also has its translates controlled by the left regular representation (again this will be made more precise later). To briefly describe the idea, consider a measure-preserving action $\Gamma\actson (X,\mu)$.  Given a set $A\subseteq X$ we can consider the subspace of $L^{2}(X,\mu)$ given by
\[\mathcal{H}_{A}=\overline{\Span\{g(\chi_{A}-\mu(A)1):g\in\Gamma\}}^{\|\cdot\|_{2}}.\]
One can then ask for sets where
\[\Gamma\actson\mathcal{H}_{A}\]
is related to representations one is more familiar with, and this provides interesting restrictions of the translates of $A$ by $\Gamma.$ For example, one could consider $A$ where $\Gamma\actson \mathcal{H}_{A}$ extends to the reduced group $C^{*}$-algebra (this is the completion of the group algebra in the left regular representation). Equivalently, for all $f\in c_{c}(\Gamma)$ we have
\[\left\|\sum_{g\in\Gamma}f(g)(\chi_{gA}-\mu(A)1)\right\|_{2}\leq \left\|\sum_{g\in\Gamma}f(g)\lambda(g)\right\|\|\chi_{A}-\mu(A)1\|_{2},\]
where $\lambda$ is the left regular representation. This says nothing in the amenable case, but in the non-amenable case implies some mixing behavior of $A.$ For example, if every measurable $A\subseteq X$ has this property and $\Gamma$ is non-amenable then the action is strongly ergodic. Based on this idea, we give a version of independence tuples, called  independence tuples satisfying the weak containment condition, where the ``representation" (via the sofic approximation) on the independence sets in question is weakly contained in the left regular representation. Since the sofic approximation is not actually a representation, we mention for clarity that we will require our independence sets to be sequences $(J_{i})_{i\geq 1}$ of subsets  of $\{1,\dots,d_{i}\}$ so that for all $f\in \CC(\Gamma),\eta>0$ we have
\[\|\sigma_{i}(f)(\chi_{J_{i}}-u_{d_{i}}(J_{1})1)\|_{2}\leq \|\lambda(f)\|\|\chi_{J_{i}}-u_{d_{i}}(J_{i})1\|_{2}+\eta\]
 for all large $i.$ Moreover, we require that the partitions
\[c\colon J_{i}\to \{1,\dots,d_{i}\}\]
are such that the pieces $c^{-1}(\{l\})$ also exhibit similarly controlled behavior by the left regular representation (albeit in a more finitary sense). Theorem 1.1 and Corollary 1.4 in \cite{Me6} indicate that the left regular representation plays a crucial role in entropy theory, and from this our strengthening of independence tuples is  natural.

 	 A priori, this different version of an independence tuple bears no relation to independence tuples developed by Kerr-Li, as we are requiring a stronger condition on the structure of the independent set but also considering less general partitions. However, using a probabilistic argument and the Sauer-Shelah Lemma we show that every  independence tuple satisfying the weak containment condition is an independence tuple. It turns out (not  surprisingly) that in the amenable case, independence tuples are  independence tuples satisfying the weak containment condition.  It is possible that independence tuples are independence tuples satisfying the weak containment condition for sofic groups, but it is not clear how one would prove this. However, we strongly believe that positivity of topological entropy is equivalent to the existence of a nondiagonal  independence pair satisfying the weak containment condition. This would be not only an analogue of  Proposition 4.16 (3) of \cite{KerrLi2}, but an  analogue of our recent results in \cite{Me6}, where it is shown (see Theorem 1.1 of \cite{Me6}) that the Koopman representation of a probability measure-preserving action with positive entropy must contain a nonzero subrepresentation of the left regular representation. The major application in our paper of independence tuples is the following question of Deninger (see \cite{DenQuest}, question 26).

\begin{question}If $\Gamma$ is a countable discrete group and $f\in M_{n}(\ZZ(\Gamma))\cap GL_{n}(\ell^{1}(\Gamma))$ which is not invertible in $M_{n}(\ZZ(\Gamma)),$  is it true that $\Det_{L(\Gamma)}(f)>1$?\end{question}

 Here $L(\Gamma)$ denotes the group von Neumann algebra, which is the strong operator topology closure of $\CC(\Gamma)$ in the left regular representation on $\ell^{2}(\Gamma)$ given by
 \[(g\xi)(h)=\xi(g^{-1}h),\mbox{ $g,h\in\Gamma.$}\]
Also $\Det_{L(\Gamma)}(f)$ is the Fuglede-Kadison determninant of $f,$ a generalization of the usual determinant in linear algebra to the infinite-dimensional setting of  operators in $M_{n}(L(\Gamma))$ see \cite{Luck} Chapter 3.2 for the precise definition. Chung-Li answered this affirmatively in Corollary 7.9 of \cite{ChungLi}  for  all amenable groups using  independence tuples. Following on the techniques in \cite{ChungLi}, David Kerr and Hanfeng Li in \cite{KerrLi2} were able to answer this in the affirmative when $\Gamma$ is residually finite. Both of these proofs use   independence tuples and their previous calculations of topological entropy for algebraic actions of residually finite groups or amenable groups. This was further exploited by Chung-Li in \cite{ChungLi} to describe algebraic actions of amenable groups with completely positive entropy. Using the main result of \cite{Me5}, and Theorem 6.8 in \cite{KerrLi2} one immediately affirmatively answers Deninger's Problem for sofic groups. However, we will be interested in generalizing this result to a larger class of $f.$ We will weaken the assumption that $f\in GL_{n}(\ell^{1}(\Gamma)).$

	To motivate our generalization, let  us  consider the case  $\Gamma=\ZZ,$ and $f\in \ZZ(\ZZ),$ and view $f$ as a Laurent polynomial. By Jensen's Formula, one can show that $\det_{L(\ZZ)}(f)>1$ if and only if $f$ has a leading coefficient of modulus one and does not have all of its roots on the unit circle. Using Fourier analysis, we see that $f$ is invertible in $\ell^{1}$ if and only if  $f$ never vanishes on  the unit circle. In particular, if $f$ is invertible in $\ell^{1}$ then $\det_{L(\ZZ)}(f)>1.$ This analysis also generalizes to any abelian group.

We note here that  the Gelfand transforms on $\ell^{1}(\ZZ)$ and $C^{*}_{\lambda}(\ZZ)$ of $f$ are both the Fourier transform, so $f$ is invertible in $\ell^{1}(\ZZ)$ if and only if $f$ is invertible in $C^{*}_{\lambda}(\ZZ)$ (equivalently $L(\ZZ)$). Consideration of the abelian case leads us to believe that it is  reasonable to expect that if $f\in M_{n}(\ZZ(\Gamma))\cap GL_{n}(L(\Gamma))$ is not invertible in $M_{n}(\ZZ(\Gamma)),$ then $\Det_{L(\Gamma)}(f)>1.$ We prove this is true in the sofic case.

\begin{theorem}\label{T:InvertGVNA} Let $\Gamma$ be a countable discrete sofic group, and $f\in M_{n}(\ZZ(\Gamma))\cap GL_{n}(L(\Gamma)).$ If $f$ is not invertible in $M_{n}(\ZZ(\Gamma)),$ then $\Det_{L(\Gamma)}(f)>1.$ \end{theorem}

For readers unfamiliar with operator algebras, we note that $f\in GL_{n}(L(\Gamma))$ is the same as saying that $f$ is invertible as a left convolution operator
\[\ell^{2}(\Gamma)^{\oplus n}\to \ell^{2}(\Gamma)^{\oplus n}.\]
We also mention in Section 4 a wide class of examples of $f\in \ZZ(\Gamma)\cap L(\Gamma)^{\times}$ as to illustrate that the above Theorem is a significant generalization of  the case $f\in \ZZ(\Gamma)\cap \ell^{1}(\Gamma)^{\times}.$ We actually prove the above Theorem by  using our results in \cite{Me5}. For notation, if $f\in M_{n}(L(\Gamma))$ we define
\[r(f)\colon \ell^{2}(\Gamma)^{\oplus n}\to \ell^{2}(\Gamma)^{\oplus n}\]
by
\[(r(f)\xi)(l)=\sum_{m=1}^{n}\sum_{g\in\Gamma}\xi(l)(g)\widehat{f_{lm}}(g)\]
if $f_{lm}=\sum_{g\in\Gamma}\widehat{f_{lm}}(g)g$ for $1\leq l,m\leq n.$ We then set
\[X_{f}=(\ZZ(\Gamma)^{\oplus n}/r(f)\ZZ(\Gamma)^{\oplus n})^{\widehat{}}.\]
Where the hat indicates that we are taking the Pontryagin dual, i.e we look at the compact, abelian group of al continuous homomorphisms from $\ZZ(\Gamma)^{\oplus n}/r(f)\ZZ(\Gamma)^{\oplus n}$ into $\TT=\RR/\ZZ$. Here we are identifying $\ZZ(\Gamma)$ inside of $\ell^{2}(\Gamma)$ via
\[\sum_{g\in\Gamma}\widehat{f}(g)g\mapsto \sum_{g\in\Gamma}\widehat{f}(g)\chi_{\{g\}}.\]
The compact, abelian group $X_{f}$ inherits a natural action of $\Gamma$ by 
\[(g\theta)(a)=\theta(g^{-1}a),\mbox{ for $\theta\in X_{f},a\in \ZZ(\Gamma)^{\oplus n}/r(f)\ZZ(\Gamma)^{\oplus n},g\in\Gamma.$}\]
The proof of Theorem \ref{T:InvertGVNA} then follows from the main result of \cite{Me5}, the following Theorem and the results of \cite{KerrLi2}.

\begin{theorem} Let $\Gamma$ be a countable discrete sofic group with sofic approximation $\sigma_{i}\colon\Gamma\to S_{d_{i}},$ and let $f\in M_{n}(\ZZ(\Gamma))\cap GL_{n}(L(\Gamma)).$ Then every $k$-tuple of points in $X_{f}$ is a $(\sigma_{i})_{i}-\IE-k$-tuple.\end{theorem}
By  Theorem 1.1 of \cite{Me5} as well as Theorem 6.8 in \cite{KerrLi2} the above Theorem implies Theorem \ref{T:InvertGVNA}.
Crucial in the proof of this theorem is both the reduction to $\ell^{2}$-independence tuples and  independence tuples satisfying the weak containment condition. If $J_{i}\subseteq\{1,\dots,d_{i}\}$ is our candidate independent set and
\[J_{i}=J_{i}^{(1)}\cup \cdots \cup J_{i}^{(k)}\]
is our candidate partition, we will need to control
\[\|\sigma_{i}(\alpha)(\chi_{J_{i}^{(s)}}-u_{d_{i}}(J_{i}^{(s)})1)\|_{2}\]
for $\alpha\in\CC(\Gamma),1\leq s\leq k.$  In particular, since we only assume that $f\in GL_{n}(L(\Gamma))$ we need to control by the norm of $\alpha$ in the left regular representation. Because of this, our modified notion of independence will be the key to proving the above theorem. Thus, we will actually show the more general fact that every $k$-tuple of points in $X_{f}$ is a independence tuple satisfying the weak containment condition.

As a consequence of our work  we have the following application to \emph{completely positive entropy}. Recall that if $\Gamma$ is a sofic group, with sofic approximation $\sigma_{i}\colon\Gamma\to S_{d_{i}},$ then $\Gamma\actson X$ where $X$ is a compact metrizable space, and $\Gamma$ acts by homeomorphisms is said to have \emph{completely positive entropy} if every nontrivial factor has positive entropy. Similarly, a probability measure-preserving action is said to have \emph{completely positive entropy} if every nontrivial (measure-theoretic) factor has positive entropy.
\begin{cor} Let $\Gamma$ be a countable discrete sofic group with sofic approximation $\sigma_{i}\colon\Gamma\to S_{d_{i}}.$ Let $f\in M_{n}(\ZZ(\Gamma))\cap GL_{n}(L(\Gamma))$ be not invertible in $M_{n}(\ZZ(\Gamma)).$ Then $\Gamma\actson X_{f}$ has completely positive topological entropy with respect to any sofic approximation. If $\Gamma$ is amenable and $\lambda_{X_{f}}$ is the Haar measure on $X_{f},$ then $\Gamma\actson (X_{f},\lambda_{X_{f}})$ has completely positive entropy as well.
\end{cor}
The amenable case uses important results from \cite{ChungLi}. For $f\in M_{n}(\ZZ(\Gamma))\cap GL_{n}(\ell^{1}(\Gamma)),$ the case of topological entropy  is a consequence of the results in \cite{KerrLi2}. The case of amenable groups and measure-theoretic entropy is contained in the results of \cite{ChungLi}, again in the situation in which $f\in M_{n}(\ZZ(\Gamma))\cap GL_{n}(\ell^{1}(\Gamma)).$ In Section 4, we will list examples of $f\in \ZZ(\Gamma)\cap L(\Gamma)^{\times}$ which are not $\ell^{1}(\Gamma)^{\times}$ when $\Gamma$ is amenable. For example, if $\Gamma$ is elementary amenable then a result of Chou in \cite{Chou} implies that $\ZZ(\Gamma)\cap L(\Gamma)^{\times}\subseteq\ell^{1}(\Gamma)^{\times}$ if and only if $\Gamma$ is virtually nilpotent. This examples reveal that, even in the amenable case, the generalization from invertbility in $\ell^{1}(\Gamma)$ to invertibility in $L(\Gamma)$ is significant. 

\textbf{Acknowledgments.} I thank the anonymous referee for their comments, which vastly improved the understandability of the paper.

\section{$\ell^{p}$-Versions of Independence Tuples}

Let us first recall the definition of a sofic group. For an $n\in \NN,$ we let $u_{n}$ be the uniform measure on $\{1,\dots,n\}.$ In general, if $A$ is a finite set, we use $u_{A}$ for the uniform probability measure on $A.$ We use $S_{n}$ for the symmetric group on $n$ letters.
\begin{definition}\emph{Let $\Gamma$ be a countable discrete group. A} sofic approximation \emph{ is a sequence of functions (not assumed to be homomorphisms) $\sigma_{i}\colon\Gamma\to S_{d_{i}}$ so that}
\[u_{d_{i}}(\{1\leq j\leq d_{i}:\sigma_{i}(gh)(j)=\sigma_{i}(g)\sigma_{i}(h)(j)\})\to 1,\mbox{\emph{ for all $g,h\in\Gamma$}}\]
\[u_{d_{i}}(\{1\leq j\leq d_{i}:\sigma_{i}(g)(j)\ne j\})\to 1,\mbox{ \emph{for all $g\in\Gamma\setminus\{e\}$}}.\]
\emph{We say that $\Gamma$ is} sofic \emph{if it has a sofic approximation.}
\end{definition}
It is known that all amenable groups and residually finite groups are sofic. Also, it is known that soficity is closed under free products  with amalgamtion over amenable subgroups (see \cite{ESZ1},\cite{LPaun},\cite{DKP},\cite{KerrDykemaPichot2}, \cite{PoppArg}).  In fact, it is shown in \cite{CHR} that graph products of sofic groups are sofic. Additonally, residually sofic groups and locally sofic groups are sofic. By Malcev's Theorem, this implies all linear groups are sofic. Finally, if $\Lambda$ is a subgroup of $\Gamma$ and $\Gamma\actson \Gamma/\Lambda$ is amenable (i.e. there is a $\Gamma$ invariant mean on $\Gamma/\Lambda$), then $\Gamma$ is sofic. For a pseudometric space $(X,\rho)$ and $A$ a finite set, and $1\leq p\leq \infty$ we define
\[\rho_{p,A}(x,y)^{p}=\frac{1}{|A|}\sum_{a\in A}\rho(x(a),y(a))^{p},\mbox{ if $p<\infty$}\]
\[\rho_{\infty,A}(x,y)=\max_{a\in A}\rho(x(a),y(a)).\]
If $A=\{1,\dots,n\}$ we shall typically use
\[\rho_{p,n}\]
instead of
\[\rho_{p,\{1,\dots,n\}}.\]
We recall the definition of sofic entropy.

\begin{definition}\emph{Let $\Gamma$ be a countable discrete group and $\Gamma\actson X$ by homeomorphisms. We say that a continuous pseudometric $\rho$ on $X$ is} dynamically generating \emph{if for all $x\ne y,$ there is a $g\in\Gamma$ so that $\rho(gx,gy)>0.$}
\end{definition}

For a pseudometric space $(X,\rho),$ subsets $A,B$ of $X$ and $\varepsilon>0$ we say that $A$ is $\varepsilon$-contained in $B$ and write $A\subseteq_{\varepsilon}B$ if for all $a\in A,$ there is a $b\in B$ with $\rho(a,b)\leq \varespilon.$ We say that $A\subseteq X$ is $\varepsilon$-dense if $X\subseteq_{\varespilon}A.$ We use $S_{\varepsilon}(X,\rho)$ for the smallest cardinality of an $\varepsilon$-dense subset of $X.$ 

\begin{definition}\emph{Let $\Gamma$ be a countable discrete sofic group with sofic approximation $\sigma_{i}\colon\Gamma\to S_{d_{i}}.$  Fix a continuous dynamically generating pseudometric $\rho$ on $X.$ For a finite $F\subseteq\Gamma,$ and $\delta>0,$ we let $\Map(\rho,F,\delta,\sigma_{i})$ be all $\phi\in X^{d_{i}}$ so that}
\[\max_{g\in F}\rho_{2,d_{i}}(\phi\circ \sigma_{i}(g),g\phi)<\delta.\]
\end{definition}

\begin{definition}\emph{ Let $\Gamma$ be a countable discrete sofic group with sofic approximation $\sigma_{i}\colon\Gamma\to S_{d_{i}}.$  Let $X$ be a compact metrizable space with $\Gamma\actson X$ by homeomorphisms. Fix a  continuous dynamically generating pseudometric $\rho$ on $X.$ Define the entropy of $\Gamma\actson X$ with respect to $\sigma_{i}$ by}
\[h_{(\sigma_{i})_{i}}(\rho,F,\delta,\varepsilon)=\limsup_{i\to\infty}\frac{1}{d_{i}}\log S_{\varepsilon}(\Map(\rho,F,\delta,\sigma_{i}),\rho_{2,d_{i}})\]
\[h_{(\sigma_{i})_{i}}(\rho,\varepsilon)=\inf_{F,\delta}h_{(\sigma_{i})_{i}}(\rho,F,\delta,\varepsilon)\]
\[h_{(\sigma_{i})_{i}}(X,\Gamma)=\sup_{\varepsilon>0}h_{(\sigma_{i})_{i}}(\rho,\varepsilon).\]
\end{definition}
In \cite{KLi} Theorem 4.5 and \cite{KLi2} Proposition 2.4, it is shown that this does not depend upon the choice of continuous dynamically generating pseudometric. In \cite{KerrLi2}, Kerr-Li defined independence tuples as a topological measure of randomness of the action, and connected it with positivity of topological entropy. One of the main results of \cite{KerrLi2} of relevance  for us is Proposition 4.16 (3), which shows that positivity of entropy is equivalent to the existence of a nondiagonal independence pair.  For our purposes, it will be convenient to consider $\ell^{q}$-versions of independence tuples.

\begin{definition}\emph{ Let $\Gamma$ be a countable discrete sofic group with sofic approximation $\sigma_{i}\colon \Gamma\to S_{d_{i}}.$ Let $X$ be a compact metrizable space with $\Gamma\actson X$ by homeomorphisms, and fix a continuous dynamically generating pseudometric $\rho$ on $X,$ and $1\leq q\leq \infty.$ Let $x=(x_{1},\dots,x_{k})\in X^{k}.$ For finite $F,K\subseteq\Gamma$ and $\delta,\varepsilon>0$ we say that a subset $J\subseteq\{1,\dots,d_{i}\}$ is a} $\ell^{q}-(\rho,F,\delta,\sigma_{i};\varepsilon,K)$-independence set for $x$ \emph{if for every $c\colon J\to \{1,\dots,k\}$ there is a $\phi\in \Map(\rho,F,\delta,\sigma_{i})$ so that}
\[\max_{g\in K}\rho_{q,J}(g\phi(\cdot),gx_{c(\cdot)})<\varepsilon.\]
\emph{We let $I_{q}(x,\rho,F,\delta,\sigma_{i};\varepsilon,K)$ be the maximum of $u_{d_{i}}(J)$ where $J$ is a $\ell^{q}-(\rho,F,\delta,\sigma_{i};\varepsilon,K)$-independence set for $x$. Additionally, we let }
\[I_{q}(x,\rho,F,\delta,(\sigma_{i})_{i};\varepsilon,K)=\limsup_{i\to\infty}I_{q}(x,\rho,F,\delta,\sigma_{i};\varepsilon,K)\]
\[I_{q}(x,\rho,(\sigma_{i})_{i};\varepsilon,K)=\inf_{\substack{\textnormal{finite } F\subseteq\Gamma,\\ \delta>0}}I_{q}(x,\rho,F,\delta,(\sigma_{i})_{i};\varepsilon,K).\]
\emph{We say that $x$ is a} $\ell^{q}-\IE-$tuple with respect to $\rho$ \emph{ if for all $\varepsilon>0,$ and finite $K\subseteq\Gamma$,}
\[I_{q}(x,\rho,(\sigma_{i})_{i};\varepsilon,K)>0.\]
\emph{We let $\IE^{k}_{(\sigma_{i})_{i},\rho}(X,\Gamma,q)$ be the set of all  $\ell^{q}-\IE-$tuples with respect to $\rho.$} \end{definition}

We shall typically denote $\IE^{k}_{(\sigma_{i})_{i},\rho}(X,\Gamma,q)$  by $\IE^{k}_{(\sigma_{i})_{i},\rho}(q)$ if $X,\Gamma$ are clear from the context.  Our goal in this section is to show that $\IE^{k}_{(\sigma_{i})_{i},\rho}(q)$ is independent of the choice of $\rho ,q,$ and that in fact $\IE^{k}_{(\sigma_{i})_{i},\rho}(q)$ is the set of independence $k$-tuples as defined by Kerr-Li in \cite{KerrLi2}. We will first show that $\IE^{k}_{(\sigma_{i})_{i},\rho}(q)$ does not depend upon $\rho.$

\begin{lemma}\label{L:whatmetric?}Let $\Gamma$ be a countable discrete sofic group and $X$ a compact metrizable space with $\Gamma\actson X$ by homeomorphisms. Let $\rho,\rho'$ be two continuous dynamically generating pseudometrics on $X.$ Then for any $1\leq q\leq\infty,$
\[\IE^{k}_{(\sigma_{i})_{i},\rho}(q)=\IE^{k}_{(\sigma_{i})_{i},\rho'}(q).\]
\end{lemma}

\begin{proof} Let $M,M'$ be the diameters of $\rho,\rho'.$  Let $x\in\IE^{k}_{(\sigma_{i})_{i},\rho'}(q).$ Let $\varepsilon>0$ and a finite $K\subseteq\Gamma$ be given. Choose a finite $K'\subseteq\Gamma$ and an $\varepsilon'>0$ so that
\[\max_{g\in K}\rho(gx,gy)<\varepsilon\]
whenever
\[\max_{g\in K'}\rho'(gx,gy)<\varepsilon'.\]
Let  $\eta'>0$ depend upon $\varepsilon$ in a manner to be determined later. Let $\alpha'>0$ be such that
\[I_{q}(x,\rho',(\sigma_{i})_{i};\eta',K')\geq \alpha'.\]

Suppose we are given a finite $F\subseteq\Gamma,$ and a $\delta>0.$ By Lemma 2.3 in \cite{KerrLi2}, we may choose a finite $F'\subseteq\Gamma$ and a $\delta'>0$ so that
\[\Map(\rho',F',\delta',\sigma_{i})\subseteq\Map(\rho,F,\delta,\sigma_{i}).\]
Let $J_{i}'$ be a $\ell^{q}-(\rho',F',\delta',\sigma_{i};\eta',K')$ independence set of maximal cardinality. Suppose we are given
\[c\colon J_{i}'\to \{1,\dots,k\},\]
choose a $\phi\in\Map(\rho',F',\delta',\sigma_{i})$ so that
\[\max_{g\in K'}\rho'_{q,J_{i}'}(g\phi(\cdot),gx_{c(\cdot)})<\eta'.\]
Let
\[C_{i}=\bigcap_{g\in K'}\{j\in J_{i}':\rho'(g\phi(j),gx_{c(j)})<\varepsilon'\}.\]
If $q<\infty,$ then
\[u_{J_{i}}(J_{i}'\setminus C_{i})\leq |K'|\left(\frac{\eta'}{\varepsilon'}\right)^{q}.\]
If $q=\infty,$ we force $\eta'<\varepsilon'$ so that $C_{i}=J_{i}'.$
For $j\in C_{i}$ we have by our choice of $\varepsilon',K'$ that
\[\rho(g\phi(j),gx_{c(j)})<\varepsilon\]
for all $g\in K.$ Thus for all $g\in K,$ and $q<\infty,$
\[\rho_{q,J_{i}'}(g\phi(\cdot),gx_{c(\cdot)})^{q}<\varepsilon^{q}+M^{q} |K'|\left(\frac{\eta'}{\varepsilon'}\right)^{q}\]
and if $q=\infty,$ then
\[\rho_{\infty,J_{i}'}(g\phi(\cdot),gx_{c(\cdot)})<\varepsilon'.\]
Choosing $\eta'>0$ sufficiently small (depending upon $K,q$) , we see that we have that $J_{i}'$ is a $(\rho,F,\delta,\sigma_{i};2\varepsilon,K)$-independence set. Thus
\[I_{q}(\rho,F,\delta,(\sigma_{i});2\varepsilon,K)\geq \alpha'.\]
As $F,\delta,\varepsilon$ are arbitrary this completes the proof.

\end{proof}

Thus for $1\leq q\leq \infty$ we will use $\IE^{k}_{(\sigma_{i})_{i}}(q)$ for $\IE^{k}_{(\sigma_{i})_{i},\rho}(q)$ for any continuous dynamically generating pseudometric $\rho.$ If $X,\Gamma$ are not clear from the context we will use
\[\IE^{k}_{(\sigma_{i})_{i}}(X,\Gamma,q)\]
instead of $\IE^{k}_{(\sigma_{i})_{i}}(q).$ It is not hard to relate our notion of combinatorial independence to that developed by Kerr-Li in \cite{KerrLi2}. We use $\IE^{k}_{(\sigma_{i})_{i}}$ for the set of $(\sigma_{i})_{i}-\IE-k$-tuples as defined by Kerr-Li in \cite{KerrLi2} (again we should really use $\IE^{k}_{(\sigma_{i})_{i}}(X,\Gamma)$ but in most cases $\Gamma\actson X$ will be clear from the context).
\begin{cor} Let $\Gamma$ be a countable discrete sofic group with sofic approximation $\sigma_{i}\colon\Gamma\to S_{d_{i}}.$ Let $X$ be a compact metrizable space with $\Gamma\actson X$ by homeomorphisms. Then $\IE^{k}_{(\sigma_{i})_{i}}(\infty)=\IE^{k}_{(\sigma_{i})_{i}}.$\end{cor}

\begin{proof}
It is easily seen that
\[\IE^{k}_{(\sigma_{i})_{i}}=\IE^{k}_{(\sigma_{i}),\rho}(\infty)\]
when $\rho$ is a compatible metric. Now apply the preceding lemma.

\end{proof}

We now show that in fact $\IE^{k}_{(\sigma_{i})_{i}}(q)$ does not depend upon $q.$ We remark that the proof is closely modeled on the proof of Proposition 4.6 in \cite{KerrLi2}. We will need Karpovsky and Milman's generalization of the Sauer-Shelah lemma (see \cite{SS1},\cite{SS2},\cite{SS3}).  For convenience we state the Lemma below.

\begin{lemma}\label{L:GSS} For any integer $k\geq 2$ and any real number $\lambda\in (k-1,k)$ there is a constant $\beta(\lambda)>0$ so that for all $n\in\NN$ if  $S\subseteq \{1,2,\dots,k\}^{n}$ has $|S|^{1/n}\geq \lambda,$ then there is an $I\subseteq\{1,\dots,n\}$ with $|I|\geq \beta(\lambda)n$ and
\[S\big|_{I}=\{1,2,\dots,k\}^{I}.\]
\end{lemma}

\begin{lemma}\label{L:whatq?} Let $\Gamma$ be a countable discrete sofic group with sofic approximation $\sigma_{i}\colon\Gamma\to S_{d_{i}}.$ Let $X$ be a compact metrizable space with $\Gamma\actson X$ by homeomorphisms. For any $1\leq q_{1},q_{2}\leq \infty$ we have
\[\IE^{k}_{(\sigma_{i})_{i}}(q_{1})=\IE^{k}_{(\sigma_{i})_{i}}(q_{2}).\]
\end{lemma}
\begin{proof}

It is clear that if $q_{1}<q_{2},$ then
\[\IE^{k}_{(\sigma_{i})_{i}}(q_{1})\supseteq \IE^{k}_{(\sigma_{i})_{i}}(q_{2}).\]
It thus suffices to prove that
\[\IE^{k}_{(\sigma_{i})_{i}}(1)\subseteq \IE^{k}_{(\sigma_{i})_{i}}(\infty).\]
Fix $k-1<\lambda<k,$ and let $\beta$ be the function defined in Lemma \ref{L:GSS}. Then we may find a  $n_{0}$ so that if $J$ is a finite set with $|J|\geq n_{0}$ and $E\subseteq \{1,\dots,k\}^{J}$ has
\[|E|\geq \lambda^{|J|},\]
then there is a $J'\subseteq J$ with
\[|J'|\geq \beta(\lambda)|J|\]
so that
\[E\big|_{J'}=\{1,\dots,k\}^{J'}.\]
Let $\rho$ be a continuous dynamically generating pseudometric on $X.$ Let
\[x=(x_{1},\dots,x_{k})\in \IE^{k}_{(\sigma_{i})_{i}}(1).\]
Let $\varepsilon>0$ and a finite $K\subseteq\Gamma$ be given. Let $\varepsilon'>0$ depend upon $\varepsilon$ in manner to be determined later. Set
\[\alpha=I_{1}(x,\rho,(\sigma_{i})_{i};\varepsilon',K).\]
Suppose we are given a finite $F\subseteq\Gamma$ and $\delta>0.$ Choose $J_{i}\subseteq\{1,\dots,d_{i}\}$ a $\ell^{1}-(\rho,F,\delta,\sigma_{i};\varepsilon',K)$-independence set for $x$ with
\[u_{d_{i}}(J_{i})=I_{1}(\rho,F,\delta,\sigma_{i};\varepsilon',K,x).\]
For every $c\colon J_{i}\to \{1,\dots,k\}$ choose a $\phi_{c}\in \Map(\rho,F,\delta,\sigma_{i})$ so that
\[\max_{g\in K}\rho_{1,J_{i}}(g\phi_{c}(\cdot),gx_{c(\cdot)})<\varepsilon'.\]
Let
\[\Xi_{c}=\bigcap_{g\in K}\{j\in J_{i}:\rho(g\phi_{c}(j),gx_{c(\cdot)})<\varepsilon\}.\]
Then
\[u_{J_{i}}(J_{i}\setminus\Xi_{c})\leq |K|\left(\frac{\varepsilon'}{\varepsilon}\right).\]
Let
\[H(t)=-t\log(t)-(1-t)\log(1-t).\]
By Stirling's formula there is a $A>0$ so that the number of subsets of $J_{i}$ of cardinality at most $|K|\left(\frac{\varepsilon'}{\varepsilon}\right)|J_{i}|$ is at most
\[A\exp\left(H\left(|K|\left(\frac{\varepsilon'}{\varepsilon}\right)\right)|J_{i}|\right)|K|\left(\frac{\varepsilon'}{\varepsilon}\right)|J_{i}|.\]
Thus there is subset $\Omega_{i}\subseteq\{1,\dots,k\}^{d_{i}}$ with
\[|\Omega_{i}|\geq \frac{k^{|J_{i}|}\exp\left(H\left(|K|\left(\frac{\varepsilon'}{\varepsilon}\right)\right)\right)^{|J_{i}|}}{A|K|\left(\frac{\varepsilon'}{\varepsilon}\right)|J_{i}|},\]
so that $\Xi_{c}$ is the same, say equal to $\Theta_{i},$ for all $c\in \Omega_{i}.$
If we choose $\varepsilon'>0$ sufficiently small then
\[|\Omega_{i}|\geq \lambda^{|J_{i}|}\]
for all large $i.$ So by our choice of $\beta$ for all large $i,$ we can find a $J_{i}'\subseteq J_{i}$ with
\[|J_{i}'|\geq \beta(\lambda)|J_{i}|\]
and
\[\Omega_{i}\big|_{J_{i}'}=\{1,\dots,k\}^{J_{i}'}.\]
Thus
\[\limsup_{i\to\infty}u_{d_{i}}(J_{i}')\geq \beta(\lambda)\limsup_{i\to\infty}u_{d_{i}}(J_{i})\geq \beta(\lambda)\alpha.\]
Choose $\varepsilon'>0$ sufficiently small so that
\[|K|\left(\frac{\varepsilon'}{\varepsilon}\right)\leq \frac{\beta(\lambda)}{2}.\]
As
\[u_{J_{i}}(J_{i}\setminus \Theta_{i})\leq \left(\frac{\varepsilon'}{\varepsilon}\right)|K|,\]
we find that
\[\liminf_{i\to\infty}u_{J_{i}}(J_{i}'\cap \Theta_{i})\geq \frac{\beta(\lambda)}{2},\]
so
\[\limsup_{i\to\infty}u_{d_{i}}(J_{i}'\cap \Theta_{i})=\limsup_{i\to\infty}\frac{|J_{i}|}{d_{i}}u_{J_{i}}(J_{i}'\cap \Theta_{i})\geq \alpha\frac{\beta(\lambda)}{2}.\]
We claim that  $J_{i}'\cap \Theta_{i}$ is a $\ell^{\infty}-(\rho,F,\delta,\sigma_{i};\varepsilon,K)$ independence set for $x$ for infinitely many $i.$  Let
\[c'\colon J_{i}'\cap \Theta_{i}\to \{1,\dots,k\}.\]
Since
\[\Omega_{i}\big|_{J_{i}'}=\{1,\dots,k\}^{J_{i}'},\]
we have
\[\Omega_{i}\big|_{J_{i}'\cap \Theta_{i}}=\{1,\dots,k\}^{J_{i}'\cap \Theta_{i}}.\]
So we may find a $c\in \Omega_{i}$ so that $c\big|_{J_{i}\cap \Theta_{i}}=c'.$ Since $\Theta_{i}=\Xi_{c}$ we have that
\[\max_{g\in K}\rho(g\phi_{c}(j),gx_{c(j)})<\varepsilon\]
for all $j\in J_{i}'\cap\Theta_{i}.$ As
\[\limsup_{i\to\infty}u_{d_{i}}(J_{i}'\cap \Theta_{i})\geq \frac{\beta(\lambda)}{2}\alpha,\]
we find that
\[\limsup_{i\to\infty}I_{\infty}(\rho,F,\delta;\varepsilon,K)\geq \frac{\beta(\lambda)}{2}\alpha.\]
Thus
\[x\in \IE^{k}_{(\sigma_{i})_{i}}(\infty).\]

\end{proof}

\section{Independence Tuples with a Weak Containment Condition}\label{S:WCC}

We now proceed to give our strengthening of independence tuples. The basic idea is instead of considering our sequence of independence sets $(J_{i})_{i\geq 1}$ to be arbitrary subsets of $\{1,\dots,d_{i}\}$ we require that in the ``representation'' $\Gamma$ has on $\ell^{2}(d_{i})$ we have that the ``subrepresentation'' generated by $(\chi_{J_{i}}-u_{d_{i}}(J_{i})1)$ is weakly contained in the left regular representation. Moreover, instead of considering arbitrary partitions
\[c\colon J_{i}\to \{1,\dots,k\}\]
we only consider one so that the pieces $J_{i,l}=c^{-1}(\{l\})$ also have the property that the ``subrepresentation'' generated by $(\chi_{J_{i,l}}-u_{d_{i}}(J_{i,l})1)$ is weakly contained in the left regular representation. The results in \cite{Me6} indicate that positivity of entropy is related in an essential way to the left regular representation. Our modified version of independence tuple is more natural from that perspective. An essential difficulty here is that since $\sigma_{i}$ is not an honest homomorphism, we do not get an honest representation this way. As we shall see shortly, one can get around this using ultraproducts. For now, we simply give a direct definition.
	
		First let us introduce some notation. For a countable discrete group $\Gamma,$ define the left regular representation
\[\lambda\colon\Gamma\to U(\ell^{2}(\Gamma))\]
by
\[(\lambda(g)\xi)(h)=\xi(g^{-1}h).\]
Extend $\lambda$ to a map $\lambda\colon \CC(\Gamma)\to B(\ell^{2}(\Gamma))$ by
\[\lambda\left(\sum_{g\in\Gamma}\alpha_{g}g\right)=\sum_{g\in\Gamma}\alpha_{g}\lambda(g).\]
For $\alpha\in \CC(\Gamma),g\in\Gamma$ set
\[\widehat{\alpha}(g)=\ip{\lambda(\alpha)\delta_{e},\delta_{g}}.\]
Then
\[\alpha=\sum_{g\in\Gamma}\widehat{\alpha}(g)g.\]

If $\Gamma$ is a sofic group with sofic approximation $\sigma_{i}\colon\Gamma\to S_{d_{i}},$ we define $\sigma_{i}\colon\CC(\Gamma)\to M_{d_{i}}(\CC)$ by
\[\sigma_{i}(f)=\sum_{g\in\Gamma}\widehat{f}(g)\sigma_{i}(g).\]
Here we are viewing $\sigma_{i}(g)$ as a permutation matrix.

\begin{definition}\label{D:regularsets}\emph{Let $\Gamma$ be a countable discrete sofic group with sofic approximation $\sigma_{i}\colon \Gamma\to S_{d_{i}}.$ For a finite $F\subseteq\CC(\Gamma),$ and a $\delta>0$ we let $\Lambda(F,\delta,\sigma_{i})$ be all $J\subseteq\{1,\dots,d_{i}\}$ so that}
\[\max_{f\in F}\|\sigma_{i}(f)(\chi_{J}-u_{d_{i}}(J)1)\|_{2}\leq \|\lambda(f)\|\|\chi_{J}-u_{d_{i}}(J)1\|_{2}+\delta.\]
\emph{We let $\Lambda_{(\sigma_{i})_{i}}$ be the set of all sequences $(J_{i})_{i\geq 1}$ with $J_{i}\subseteq\{1,\dots,d_{i}\}$ so that for every finite $F\subseteq\CC(\Gamma),$ and for every $\delta>0$ it is true that for all large $i,$ $J_{i}\in \Lambda(F,\delta,\sigma_{i}).$}
\end{definition}

We now mention the formalization via ultrafilters. Fix a free ultrafilter $\omega\in\beta\NN\setminus \NN.$ Let
\[A=\frac{\left\{(f_{i})_{i=1}^{\infty}:f_{i}\in \ell^{\infty}(d_{i}),\sup_{i}\|f_{i}\|_{\infty}<\infty\right\}}{\left\{(f_{i})_{i=1}^{\infty}\colon f_{i}\in \ell^{\infty}(d_{i}),\sup_{i}\|f_{i}\|_{\infty}<\infty,\lim_{i\to \omega}\|f_{i}\|_{\ell^{2}(d_{i},u_{d_{i}})}=0\right\}}.\]
For a sequence $f_{i}\in \ell^{\infty}(d_{i},u_{d_{i}})$ we use $(f_{i})_{i\to\omega}$ for the image in $A$ of $(f_{i})_{i\geq 1}$ under the quotient map. There is a well-defined inner product on $A$ given by
\[\ip{(f_{i})_{i\to\omega},(g_{i})_{i\to\omega}}=\lim_{i\to\omega}\frac{1}{d_{i}}\sum_{j=1}^{d_{i}}f_{i}(j)\overline{g_{i}(j)}.\]
Let $L^{2}(A,u_{\omega})$ be the completion of $A$ under this inner-product. Then we have a well-defined unitary representation
\[\sigma_{\omega}\colon\Gamma\to \mathcal{U}(L^{2}(A,u_{\omega}))\]
defined densely by
\[\sigma_{\omega}(g)(f_{i})_{i\to\omega}=(f_{i}\circ \sigma_{i}(g))_{i\to\omega}\]
if $f_{i}\in \ell^{\infty}(d_{i},u_{d_{i}})$ and
\[\sup_{i}\|f_{i}\|_{\infty}<\infty.\]
We then see that $\Lambda_{(\sigma_{i})_{i}}$ can be regarded as all sequences $J_{i}$ of subsets of $\{1,\dots,d_{i}\}$ so that if $\omega\in\beta\NN\setminus\NN$ is any free ultrafilter and we set
\[\chi_{J_{\omega}}^{o}=(\chi_{J_{i}}-u_{d_{i}}(J_{i})1)_{i\to\omega},\]
\[\mathcal{K}=\overline{\Span\{\sigma_{\omega}(g)\chi_{J_{\omega}}^{o}:g\in\Gamma\}},\]
then the representation $\Gamma\actson \mathcal{K}$ is weakly contained in the left regular representation (see \cite{BHV} Appendix F for the relevant facts about weak containment of representations).

\begin{definition}\emph{Let $\Gamma$ be a countable discrete sofic group and $\sigma_{i}\colon\Gamma\to S_{d_{i}}$ a sofic approximation. Let $X$ be a compact metrizable space with $\Gamma\actson X$ by homeomorphisms. Let $\rho$ be a continuous dynamically generating pseudometric on $X.$  Let $1\leq q\leq \infty,$ and $x\in X^{k}.$ Fix finite $K,F\subseteq\Gamma,E\subseteq \CC(\Gamma)$ and $\varepsilon,\delta,\eta>0.$  We say that a sequence $(J_{i})_{i}$ is a} $\Lambda-\ell^{q}-(\rho,F,\delta,E,\eta;\varepsilon,K)$-independence sequence for $x$\emph{ if $(J_{i})_{i}\in \Lambda_{(\sigma_{i})_{i}},$ and for all $i,$  for all $c\colon J_{i}\to \{1,\dots,k\}$ with $c^{-1}(\{l\})\in \Lambda(E,\eta,\sigma_{i}),$ there is a $\phi\in\Map(\rho,F,\delta,\sigma_{i})$ with}
\[\max_{g\in K}\rho_{q,J_{i}}(g\phi(\cdot),gx_{c(\cdot)})<\vaerpsilon.\]
\emph{ We let $I_{\Lambda,q}(x,\rho,F,\delta,E,\eta,(\sigma_{i})_{i};\varepsilon,K)$ be the supremum of
\[\limsup_{i\to\infty}u_{d_{i}}(J_{i})\]
 over all sequences $(J_{i})$ which are $\Lambda-\ell^{q}-(\rho,F,\delta,E,\eta;\vaerpsilon,K)$-independence sequences. We then set}
 \[I_{\Lambda,q}(x,\rho,F,\delta,(\sigma_{i})_{i};\varepsilon,K)=\sup_{\substack{\textnormal{ finite } E\subseteq\CC(\Gamma),\\ \eta>0}}I_{\Lambda,q}(x,\rho,F,\delta,E,\eta;\varepsilon,K),\]
 \[I_{\Lambda,q}(x,\rho,(\sigma_{i})_{i};\varepsilon,K)=\inf_{\substack{\textnormal{ finite} F\subseteq\Gamma,\\ \delta>0}}I_{\Lambda,q}(x,\rho,F,\delta;\varepsilon,K).\]

 \end{definition}

\begin{definition}\emph{Let $\Gamma$ be a countable discrete group and $\sigma_{i}\colon\Gamma\to S_{d_{i}}$ a sofic approximation. Let $X$ be a compact metrizable space with $\Gamma\actson X$ by homeomorphisms. Let $\rho$ be a continuous dynamically generating pseudometric on $X,$ and $1\leq q<\infty.$  We say that $x=(x_{1},\dots,x_{k})$ is a}  $\ell^{q}-(\sigma_{i})-\IE-k$-tuple satisfying the weak containment condition (or a $\Lambda_{(\sigma_{i})_{i}}-\IE-k$-tuple)  \emph{if for every $\varepsilon>0$ and $K\subseteq\Gamma$ finite}
\[I_{\Lambda,q}(x,\rho,(\sigma_{i})_{i};\varepsilon,K )>0.\]
\emph{We use $\IE^{\Lambda,k}_{(\sigma_{i}),\rho}(X,\Gamma,q)$ for the set of  $\ell^{q}-(\sigma_{i})-IE-k$-tuples satisfying the weak containment condition. If $X,\Gamma$ are clear from context we will simply use $\IE^{\Lambda,k}_{(\sigma_{i}),\rho}(q).$}
\end{definition}

In fact, following the proof of Lemma \ref{L:whatmetric?}, one shows that $\IE^{\Lambda,k}_{(\sigma_{i}),\rho}(q)$ is independent of $\rho,$ so we simply use $\IE^{\Lambda,k}_{(\sigma_{i})_{i}}(q).$ However, we do not know if $\IE^{\Lambda,k}_{(\sigma_{i})_{i}}(q)$ is independent of $q.$ Note that if $M$ is the diameter of $(X,\rho)$ then by standard H\"{o}lder estimates we have for any finite set $A$ and $1\leq q_{1}<q_{2}<\infty:$
\[\rho_{q_{1},A}(x,y)\leq \rho_{q_{2},A}(x,y)\leq M^{1-\frac{q_{1}}{q_{2}}}\rho_{q_{1},A}(x,y)^{\frac{q_{1}}{q_{2}}},\mbox{ for any $x,y\in X^{A}.$}\]
From this it is not hard to see that 
\[\IE^{\Lambda,k}_{(\sigma_{i})_{i}}(q_{1})=\IE^{\Lambda,k}_{(\sigma_{i})_{i}}(q_{2})\]
for all $1\leq q_{1},q_{2}<\infty.$

	The definition may seem a little ad hoc. The following proposition will hopefully make it seem more natural. Essentially, this proposition will tell us two things: first if we fix a $\alpha>0,$ and choose a subset $J_{i}\subseteq\{1,\dots,d_{i}\}$ of size roughly $\alpha d_{i}$  uniformly at random, then $(J_{i})_{i}$ will be in $\Lambda_{(\sigma_{i})}$ with high probability. Secondly, suppose we are given a sequence $(J_{i})_{i}$ in $\Lambda_{(\sigma_{i})}$ a finite $E\subseteq\Gamma$ and a $\eta>0,$ and a probability measure $\mu$ on $\{1,\dots,k\}.$ If  we choose a partition of $(J_{i})_{i}$ into sets of size $\mu(\{1\})u_{d_{i}}(J_{i}),\dots,\mu(\{k\})u_{d_{i}}(J_{i})$ uniformly at random, then with high probability, each of the pieces of the partition will be in $\Lambda(E,\eta,\sigma_{i}).$ Thus we may view independence tuples satisfying the weak containment condition as simply a randomization of independence tuples as defined by Kerr-Li in \cite{KerrLi2}. The proof is a simple adaption of Bowen's argument for the computation of sofic entropy of Bernoulli shifts in \cite{Bow}.

\begin{proposition}\label{P:randomization} Let $\Gamma$ be a countable discrete sofic group, with sofic approximation $\sigma_{i}\colon\Gamma\to S_{d_{i}}.$ Let $\mu$ be a probability measure on $\{1,\dots,k\}.$ Fix a sequence $(J_{i})_{i}\in \Lambda_{(\sigma_{i})_{i}}.$ Then, for any $E\subseteq\CC(\Gamma)$ finite, and any $\eta>0$ and any $1\leq l\leq k$ we have
\[\mu^{\otimes J_{i}}(\{p\in \{1,\dots,k\}^{J_{i}}:p^{-1}(\{l\})\in \Lambda(E,\eta,\sigma_{i})\})\to 1,\]
\[\mu^{\otimes J_{i}}(\{p\in\{1,\dots,k\}^{J_{i}}:|u_{d_{i}}(p^{-1}(\{l\}))-\mu(\{l\})u_{d_{i}}(J_{i})|>\eta\})\to 0.\]
\end{proposition}

\begin{proof}

As our claim is probabilistic, we may assume $E=\{f\}.$ We make the following two claims.

\emph{Claim 1:} For all $F\subseteq\Gamma\setminus\{e\}$ finite, for every $1\leq l\leq k,$ for every $\delta>0,$ we have
\[\mu^{\otimes J_{i}}\left(\left\{p\in \{1,\dots,k\}^{J_{i}}:\left|u_{d_{i}}(\sigma_{i}(g)p^{-1}(\{l\})\cap p^{-1}(\{l\}))-u_{d_{i}}(J_{i}\cap \sigma_{i}(g)J_{i})\mu(\{l\})^{2}\right|>\delta\mbox{ for some $g\in F$}\right\}\right)\to 0.\]
\emph{Claim 2}: For every $1\leq l\leq k,$ for every $\delta>0,$ we have\[\mu^{\otimes J_{i}}(\{p\in \{1,\dots,k\}^{J_{i}}:|u_{d_{i}}(p^{-1}(\{l\}))-\mu(\{l\})u_{d_{i}}(J_{i})|>\delta\})\to 0.\]

Suppose we accept the two claims. Then, we may find a subset $P_{i}\subseteq \{1,\dots,k\}^{J_{i}}$ so that for every sequence $p_{i}\in P_{i},$
\[u_{d_{i}}(p_{i}^{-1}(\{l\}))-\mu(\{l\})u_{d_{i}}(J_{i})\to 0,\]
\[|u_{d_{i}}(\sigma_{i}(g)p_{i}^{-1}(\{l\})\cap p_{i}^{-1}(\{l\}))-\mu(\{l\})^{2}u_{d_{i}}(J_{i}\cap \sigma_{i}(g)J_{i})|\to 0,\mbox{ for every $g\in\Gamma\setminus\{e\}$}\]
and
\[\mu^{\otimes J_{i}}(P_{i})\to 1.\]

Fix a sequence $p_{i}\in P_{i}.$ Let $\xi_{l}=\chi_{p_{i}^{-1}(\{l\})}-u_{d_{i}}(p_{i}^{-1}(\{l\}))1,$ $\zeta=\chi_{J_{i}}-u_{d_{i}}(J_{i})1.$ We use $o(1)$ for any expression which goes to zero as $i\to\infty.$ Then for any $f\in \CC(\Gamma)$ and $i\in \NN$
\begin{align}\label{E:aldgakl}
\|\sigma_{i}(f)\xi_{l}\|_{2}^{2}&=\sum_{g,h\in\Gamma}\widehat{f}(g)\overline{\widehat{f}(h)}\ip{\sigma_{i}(h)^{-1}\sigma_{i}(g)\xi_{l},\xi_{l}}\\ \nonumber
&=o(1)+\sum_{g,h\in\Gamma}\widehat{f}(g)\overline{\widehat{f}(h)}\ip{\sigma_{i}(h^{-1}g)\xi_{l},\xi_{l}}\\ \nonumber
&=o(1)+\sum_{g\in\Gamma}\widehat{f^{*}f}(g)\ip{\sigma_{i}(g)\xi_{l},\xi_{l}}\\ \nonumber
\end{align}
We have that 
\begin{equation}\label{E:aldgja}
\ip{\sigma_{i}(e)\xi_{l},\xi_{l}}=o(1)+\|\xi_{l}\|_{2}^{2}=o(1)+u_{d_{i}}(p_{i}^{-1}(\{l\}))-u_{d_{i}}(p_{i}^{-1}(\{l\}))^{2}=o(1)+\mu(\{l\})u_{d_{i}}(J_{i})-\mu(\{l\})^{2}u_{d_{i}}(J_{i})^{2}.
\end{equation}
and for $g\ne e$ we have
\begin{align*}
\ip{\sigma_{i}(g)\xi_{l},\xi_{l}}&=u_{d_{i}}(\sigma_{i}(g)p_{i}^{-1}(\{l\})\cap p_{i}^{-1}(\{l\}))-u_{d_{i}}(p_{i}^{-1}(\{l\}))^{2}\\
&=o(1)+\mu(\{l\})^{2}u_{d_{i}}(J_{i}\cap \sigma_{i}(g)J_{i})-\mu(\{l\})^{2}u_{d_{i}}(J_{i})^{2}\\
&=o(1)+\mu(\{l\})^{2}\ip{\sigma_{i}(g)\zeta,\zeta}.\\
\end{align*}
Additionally
\begin{equation}\label{E:slkagkgj}
\|\zeta\|_{2}^{2}=u_{d_{i}}(J_{i})-u_{d_{i}}(J_{i})^{2}.
\end{equation}
Combining $(\ref{E:aldgakl}),(\ref{E:aldgja}),(\ref{E:slkagkgj})$ we see that 
\begin{align*}
\|\sigma_{i}(f)\xi_{l}\|_{2}^{2}&=o(1)+\widehat{f^{*}f}(e)(\mu(\{l\})u_{d_{i}}(J_{i})-\mu(\{l\})^{2}u_{d_{i}}(J_{i})^{2})\\ \nonumber
&         +\mu(\{l\})^{2}\sum_{g\in\Gamma\setminus\{e\}}\widehat{f^{*}f}(g)\ip{\sigma_{i}(g)\zeta,\zeta}\\
&=o(1)+\widehat{f^{*}f}(e)(\mu(\{l\})-\mu(\{l\})^{2})u_{d_{i}}(J_{i})\\
&+\mu(\{l\})^{2}\sum_{g\in\Gamma}\widehat{f^{*}f}(g)\ip{\sigma_{i}(g)\zeta,\zeta}.\\
\end{align*}
By the same logic,
\[\|\sigma_{i}(f)\zeta\|_{2}^{2}=o(1)+\sum_{g\in\Gamma}\widehat{f^{*}f}(g)\ip{\sigma_{i}(g)\zeta,\zeta}.\]
Thus
\[\|\sigma_{i}(f)\xi_{l}\|_{2}^{2}=o(1)+\widehat{f^{*}f}(e)(\mu(\{l\})-\mu(\{l\})^{2})u_{d_{i}}(J_{i})+\mu(\{l\})^{2}\|\sigma_{i}(f)\zeta\|_{2}^{2}.\]
Since $J_{i}\in \Lambda_{(\sigma_{i})_{i}},$
\[\|\sigma_{i}(f)\zeta\|_{2}^{2}\leq \|\lambda(f)\|^{2}(u_{d_{i}}(J_{i})-u_{d_{i}}(J_{i})^{2})+\eta\]
for all large $i.$ Thus for all large $i,$
\begin{align*}
\|\sigma_{i}(f)\xi_{l}\|_{2}^{2}&\leq o(1)+\eta+\mu(\{l\})^{2}\|\lambda(f)\|^{2}(u_{d_{i}}(J_{i})-u_{d_{i}}(J_{i})^{2})\\
&+\|\lambda(f)\|^{2}(\mu(\{l\})-\mu(\{l\})^{2})u_{d_{i}}(J_{i})\\
&= o(1)+\eta+\|\lambda(f)\|^{2}(\mu(\{l\})u_{d_{i}}(J_{i})-\mu(\{l\})^{2}u_{d_{i}}(J_{i})^{2}).
\end{align*}
Since
\[|\|\xi_{l}\|_{2}^{2}-(\mu(\{l\})u_{d_{i}}(J_{i})-\mu(\{l\})^{2}u_{d_{i}}(J_{i})^{2})|\to 0,\]
and $\eta$ is arbitrary this proves the proposition.

We thus turn to the proof of Claim 1 and Claim 2. For Claim 1, it suffices to assume $F=\{g\}.$ We have

\[\int u_{d_{i}}(\sigma_{i}(g)p^{-1}(\{l\})\cap p^{-1}(\{l\}))\,d\mu^{\otimes J_{i}}(p)=\frac{1}{d_{i}}\sum_{j=1}^{d_{i}}\int \chi_{p^{-1}(\{l\})}(j)\chi_{p^{-1}(\{l\})}(\sigma_{i}(g)^{-1}(j))\,d\mu^{\otimes J_{i}}(p).\]
Note that $\chi_{p^{-1}(\{l\})}(j)\chi_{p^{-1}(\{l\})}(\sigma_{i}(g)^{-1}(j))$ can only be positive if $j\in \sigma_{i}(g)J_{i}\cap J_{i}.$ Thus the above sum is
\[\frac{1}{d_{i}}\sum_{j\in \sigma_{i}(g)J_{i}\cap J_{i}}\int \chi_{\{l\}}(p(j))\chi_{\{l\}}(p(\sigma_{i}(g)^{-1}(j)))\,d\mu^{\otimes J_{i}}(p).\]
Since
\[u_{d_{i}}(\{1\leq j\leq d_{i}:\sigma_{i}(g)(j)\ne j\})\to 1,\]
 we have that
\[\frac{1}{d_{i}}\sum_{j\in \sigma_{i}(g)J_{i}\cap J_{i}}\int \chi_{\{l\}}(p(j))\chi_{\{l\}}(p(\sigma_{i}(g)^{-1}(j)))\,d\mu^{\otimes J_{i}}(p)=u_{d_{i}}(\sigma_{i}(g)J_{i}\cap J_{i})\mu(\{l\})^{2}+o(1).\]
By Chebyshev's inequality, it thus suffices to show that
\[\int u_{d_{i}}(\sigma_{i}(g)p^{-1}(\{l\})\cap p^{-1}(\{l\}))^{2}\,d\mu^{\otimes J_{i}}(p)=u_{d_{i}}(\sigma_{i}(g)J_{i}\cap J_{i})^{2}\mu(\{l\})^{4}+o(1).\]

For this, we have

\[\int u_{d_{i}}(\sigma_{i}(g)p^{-1}(\{l\})\cap p^{-1}(\{l\}))^{2}\,d\mu^{\otimes J_{i}}(p)=\]
\[\frac{1}{d_{i}^{2}}\sum_{j,k\in \sigma_{i}(g)J_{i}\cap J_{i}}\int \chi_{\{l\}}(p(j))\chi_{\{l\}}(p(\sigma_{i}(g)^{-1}(j)))\chi_{\{l\}}(p(k))\chi_{\{l\}}(p(\sigma_{i}(g)^{-1}(k)))\,d\mu^{\otimes J_{i}}(p).\]

We claim that
\begin{equation}\label{E:distinctionyo}
u_{d_{i}}\otimes u_{d_{i}}(\{(j,k):|\{j,k,\sigma_{i}(g)^{-1}(j),\sigma_{i}(g)^{-1}(k)\}|=4\})\to 1,\mbox{ as $i\to\infty$}.
\end{equation}
We already know that
\[u_{d_{i}}\otimes u_{d_{i}}(\{(j,k):|\{j,k,\sigma_{i}(g)^{-1}(j)\ne j,\sigma_{i}(g)^{-1}(k)\ne k\}|\})\to 1,\mbox{ as $i\to\infty$}.\]
 Additionally,
 \[u_{d_{i}}\otimes u_{d_{i}}(\{(j,k):j\ne k\})\to 1,\mbox{ as $i\to\infty$}.\]
 Thus it suffices to show that 
\[u_{d_{i}}\otimes u_{d_{i}}(\{(j,k):j\ne k,\sigma_{i}(g)^{-1}(j)\ne \sigma_{i}(g)^{-1}(k),|\{j,k,\sigma_{i}(g)^{-1}(j),\sigma_{i}(g)^{-1}(k)\}|<4\})\to 0.\]
Suppose then that $j\ne k,\sigma_{i}(g)^{-1}(j)\ne j,\sigma_{i}(g)^{-1}(k)\ne k.$ Then, $\sigma_{i}(g)^{-1}(j)\ne \sigma_{i}(g)^{-1}(k).$ So 
\[|\{j,k,\sigma_{i}(g)^{-1}(j),\sigma_{i}(g)^{-1}(k)\}|<4\]
 if and only if $j=\sigma_{i}(g)^{-1}(k)$ or $k=\sigma_{i}(g)^{-1}(j).$ However the union of $\{(j,k):\sigma_{i}(g)^{-1}(k)=j\},\{(j,k):k=\sigma_{i}(g)^{-1}(j)\}$ has cardinality at most $2d_{i}.$ This proves (\ref{E:distinctionyo}). So

\[\int u_{d_{i}}(\sigma_{i}(g)p^{-1}(\{l\})\cap p^{-1}(\{l\}))^{2}\,d\mu^{\otimes J_{i}}(p)=o(1)+u_{d_{i}}(\sigma_{i}(g)J_{i}\cap J_{i})^{2}\mu(\{l\})^{4}.\]
This proves Claim 1.

The proof of Claim 2 is similar, and in fact has already been done by Bowen in \cite{Bow} Theorem 8.1, it can also be seen as a consequence of the law of large numbers.
\end{proof}

We now show that the set of $\ell^{q}$-independence tuples satisfying the weak containment condition is contained in the set of $\ell^{q}$-independence tuples.

\begin{proposition}\label{P:regularindependence}
Let $\Gamma$ be a countable discrete group with sofic approximation $\sigma_{i}\colon\Gamma\to S_{d_{i}}.$ Let $X$ be a  compact metrizable space with $\Gamma\actson X$ by homeomorphisms. Then for any $1\leq q\leq \infty,$
\[\IE^{\Lambda,k}_{(\sigma_{i})_{i}}(q)\subseteq\IE^{k}_{(\sigma_{i})_{i}}(q).\]
\end{proposition}

\begin{proof}

Fix a compatible metric $\rho$ on $X.$  Let $x=(x_{1},\dots,x_{k})\in \IE^{\Lambda,k}_{(\sigma_{i})_{i}}(q).$ Let $\varepsilon>0,$ and $K\subseteq\Gamma$ finite be given.  Set
\[\alpha=I_{\Lambda,q}(x,\rho,(\sigma_{i})_{i};\varepsilon,K).\]
Fix $k-1<\lambda<k$ and let $\beta(\lambda)$ be as in the Sauer-Shelah Lemma.

Suppose we are given a finite $F\subseteq\Gamma,$ and $\delta>0.$ Choose a finite $E\subseteq\CC(\Gamma),$ and a $\eta>0$ so that
\[I_{\Lambda}(x,\rho,F,\delta,E,\eta,(\sigma_{i})_{i};\varepsilon,K)\geq \frac{\lambda}{k}\alpha.\]
 Let $(J_{i})_{i=1}^{\infty}\in \Lambda_{(\sigma_{i})}$ be a $\ell^{q}-(\rho,F,\delta,E,\eta)$-independence sequence for $x$ with
\[\limsup_{i\to \infty}u_{d_{i}}(J_{i})\geq \frac{\lambda}{k}\alpha.\]
Let $\Lambda_{k}(E,\eta,J_{i})$ be the set of all $c\colon J_{i}\to \{1,\dots,k\}$ so that $c^{-1}(\{l\})\in \Lambda(E,\eta,\sigma_{i})$ for $1\leq l\leq k.$ By Proposition \ref{P:randomization} we have
\[\lim_{i\to \infty}\frac{|\Lambda_{k}(E,\eta,J_{i})|}{k^{|J_{i}|}}=1.\]
 So by Lemma \ref{L:GSS}, for all large $i$ we can find $J_{i}'\subseteq J_{i}$ with
\[\Lambda_{k}(E,\eta,J_{i})\big|_{J_{i}'}=\{1,\dots,k\}^{J_{i}'}\]
and
\[u_{d_{i}}(J_{i}')\geq \beta(\lambda)u_{d_{i}}(J_{i}).\]
We claim that $J_{i}'$ is a $\ell^{q}-(\rho,F,\delta,\sigma_{i})$-independence set for $x$ for all large $i.$

For this, let $c'\colon J_{i}'\to \{1,\dots,k\}.$ Then, there is a $c\in \Lambda_{k}(E,\eta,J_{i})$ so that
\[c\big|_{J_{i}'}=c'.\]
By the definition of $(\rho,F,\delta,E,\eta,(\sigma_{i})_{i})$ independence, there is a $\phi\in\Map(\rho,F,\delta,\sigma_{i})$ so that
\[\max_{g\in K}\rho_{q,J_{i}}(g\phi(\cdot),gx_{c(\cdot)})<\varepsilon.\]
As $c\big|_{J_{i}'}=c',$ we find that
\[\max_{g\in K}\rho_{q,J_{i}'}(g\phi(\cdot),gx_{c'(\cdot)})<\frac{\varepsilon}{\beta(\lambda)}.\]
As
\[\limsup_{i\to\infty}u_{d_{i}}(J_{i}')\geq \beta(\lambda)\alpha,\]
we see that
\[I_{q}\left(\rho,F,\delta,(\sigma_{i});\frac{\varepsilon}{\beta(\lambda)},K\right)\geq \beta(\lambda)\alpha.\]
Taking the infimum over all $F,\delta$ completes the proof.

\end{proof}

We now discuss the analogue of Proposition 4.16 from \cite{KerrLi2} for independence tuples with a weak containment condition. Recall that if $X,Y$ are compact metrizable spaces and $\Gamma\actson X, \Gamma\actson Y$ by homeomoprhisms, then a continuous, $\Gamma$-equivariant, surjection $\pi\colon X\to Y$ is called a \emph{factor map}. If there is a factor map $\pi\colon X\to Y,$ we call $Y$ a \emph{factor} of $X.$ 

\begin{proposition}\label{P:permanence} Let $\Gamma$ be a countable discrete sofic group with sofic approximation $\sigma_{i}\colon\Gamma\to S_{d_{i}}.$ Fix $1\leq q\leq \infty.$ Let $X$ be a compact metrizable space with $\Gamma\actson X$ by homeomorphisms.

(1): If $\IE^{\Lambda,2}_{(\sigma_{i})_{i}}(q)\setminus\{(x,x):x\in X\}$ is nonempty, then $h_{(\sigma)_{i}}(X,\Gamma)>0.$

(2): We have that $\IE^{\Lambda,k}_{(\sigma_{i})_{i}}(q)$ is a closed $\Gamma$-invariant subset of $X^{k},$ where $\Gamma\actson X^{k}$ is the product action.

(3): Let $Y$ be a compact metrizable space with $\Gamma\actson Y$ by homeomorphisms and $\pi\colon X\to Y$ a factor map. Then
\[\pi^{k}(\IE^{\Lambda,k}_{(\sigma_{i})_{i}}(q,X,\Gamma))\subseteq\IE^{\Lambda,k}_{(\sigma_{i})_{i}}(q,Y,\Gamma).\]

(4): Suppose that $Z$ is a closed $\Gamma$-invariant subset of $X,$ then $\IE^{\Lambda,k}_{(\sigma_{i})_{i}}(Z,\Gamma)\subseteq\IE^{\Lambda,k}_{(\sigma_{i})_{i}}(X,\Gamma).$

\end{proposition}

\begin{proof}

(1): This is a consequence of the preceding Proposition and Proposition 4.16 (3) in \cite{KerrLi2}.

(2): Fix a compatible metric $\rho$ on $X$ and $g\in\Gamma.$ Let $\alpha_{g}(x)=gx.$ Then for any finite $F\subseteq\Gamma,$ for any $\delta>0,$ there is a $\delta'>0$ so that if
\[\phi\in \Map(\rho,\{g^{-1}\}\cup (g^{-1}F)\cup\{g\},\delta',\sigma_{i})\]
then
\[\alpha_{g}\circ \phi\circ \sigma_{i}(g)^{-1}\in \Map(\rho,F,\delta,\sigma_{i}),\]
for all large $i.$ Thus,
\[\IE^{\Lambda,k}_{(\sigma_{i})_{i}}\]
is $\Gamma$-invariant. The fact that it is closed is a trivial consequence of the definitions.

(3): Let $\rho,\rho'$ be compatible metrics on $X,Y.$ Let $M,M'$ be the diameter of $\rho,\rho'.$ Suppose we are given a $\varepsilon'>0,$ and let $\eta'>0$ depend upon $\varepsilon$ to be determined shortly. Choose a $\varepsilon>0$ so that
\[\rho(x,y)<\varepsilon\]
implies
\[\rho'(\pi(x),\pi(y))<\eta'.\]
Let $\eta>0$ depend upon $\varepsilon$ in a manner to be determined later.
Given a finite $F'\subseteq\Gamma$ finite and a $\delta'>0$ we can find a finite $F\subseteq\Gamma$ and a $\delta>0$ so that
\[\pi^{d_{i}}(\Map(\rho,F,\delta,\sigma_{i}))\subseteq\Map(\rho',F',\delta',\sigma_{i}),\]
(this follows by the same argument in Lemma 2.3 of \cite{KerrLi2}).
Let  $x\in X^{k},$ and let $J_{i}$ be a $(x,\rho,F,\delta;\eta,\{e\})$-independence set, and suppose we are given
\[c\colon J_{i}\to \{1,\dots,k\}.\]
Choose $\phi\in\Map(\rho,F,\delta,\sigma_{i})$ so that
\[\rho_{q,J_{i}}(\phi,x_{c(\cdot)})<\eta.\]
Then
\[u_{J_{i}}(\{j\in J_{i}:\rho(\phi(j),x_{c(j)})<\varepsilon\})\geq (1-\frac{\eta^{q}}{\varepsilon^{q}}).\]
By our choice of $\varepsilon,$
\[\rho_{q}(\pi\circ \phi,\pi(x_{c(\cdot)}))^{q}\leq (\eta')^{q}+M\left(\frac{\eta^{q}}{\varepsilon^{q}}\right).\]
Choosing $\eta,\eta'$ appropriately we have that $J_{i}$ is a $(\pi(x),\rho,F',\delta';\varepsilon',\{e\})$ independence set.

(4): This is trivial.

\end{proof}

We now proceed to show that $\ell^{q}-\Lambda_{(\sigma_{i})}$-tuples are the same as $\ell^{q}$-independence tuples in the amenable case. For this, we will need the following general fact: if $\Gamma$ is an amenable group and $\pi\colon\Gamma\to \mathcal{U}(\mathcal{H})$ is a unitary representation, $\pi$ is weakly contained in $\lambda.$ See \cite{BHV} Theorem G.3.2 for a proof.

\begin{proposition}\label{P:automaticLR} Let $\Gamma$ be a countable discrete amenable group. Let $\sigma_{i}\colon\Gamma\to S_{d_{i}}$ be a sofic approximation. Then every sequence of subsets of $\{1,\dots,d_{i}\}$ is in $\Lambda_{(\sigma_{i})_{i}}.$
\end{proposition}

\begin{proof}

Automatic from the remarks about ultraproducts and weak containment following Definition \ref{D:regularsets}.

\end{proof}

\begin{proposition}\label{T:amenablecase} Let $\Gamma$ be a countable discrete amenable group. Let $\sigma_{i}\colon\Gamma\to S_{d_{i}}$ be a sofic approximation. Let $X$ be a compact metrizable space with $\Gamma\actson X$ by homeomorphisms. Then for any $1\leq q\leq\infty,$
\[\IE_{(\sigma_{i})_{i}}^{k}(q)=\IE_{(\sigma_{i})_{i}}^{\Lambda,k}(q).\]
\end{proposition}

\begin{proof}

Fix a compatible metric $\rho$ on $X.$
By Proposition \ref{P:automaticLR} we have
\[\IE_{(\sigma_{i})_{i}}^{k}(q)\supseteq \IE_{(\sigma_{i})_{i}}^{\Lambda,k}(q).\]
Conversely, let $(x_{1},\dots,x_{k})$ be  in $\IE_{(\sigma_{i})_{i}}^{k}$ but not in $\IE_{(\sigma_{i})_{i}}^{\Lambda,k}.$ Choose a $\varepsilon>0$ and a finite $K\subseteq\Gamma$ so that
\[I_{\Lambda,q}(x,\rho,(\sigma_{i})_{i};\varepsilon,K)=0.\]
Since
\[I_{q}(x,\rho,(\sigma_{i})_{i};\varepsilon,K)>0,\]
we can find a finite $F\subseteq\Gamma$ and a $\delta>0$ so that
\[I_{\Lambda,q}(x,\rho,F,\delta,(\sigma_{i})_{i};\varepsilon,K)<I_{q}(x,\rho,F,\delta,(\sigma)_{i};\varepsilon,K).\]
Choose  a finite $E\subseteq\Gamma,$ and a $\eta>0$ so that
\[I_{\Lambda,q}(x,\rho,F,\delta,E,\eta,(\sigma_{i})_{i};\varepsilon,K)<I_{q}(x,\rho,F,\delta,(\sigma_{i})_{i};\varepsilon,K).\]
Choose $(J_{i})_{i\geq 1}\in\Lambda_{(\sigma_{i})_{i}}$ so that  $(J_{i})_{i\geq 1}$ is a $\Lambda-\ell^{q}-(\rho,F,\delta,E,\eta;\vaerpsilon,K)$ independence sequence  with
\[\limsup_{i\to\infty}u_{d_{i}}(J_{i})=I_{\Lambda}(x,\rho,F,\delta,\sigma_{i};\varespilon,K).\]
Since
\[I_{\Lambda,q}(x,\rho,F,\delta,E,\eta,(\sigma_{i})_{i};\varepsilon,K)<I_{q}(x,\rho,F,\delta,(\sigma_{i})_{i};\varepsilon,K)=\limsup_{i\to\infty}u_{d_{i}}(J_{i})\]
we can find a subsequence $i_{l},$ and a partition
\[J_{i_{l}}=J_{i_{l}}^{(1)}\cup\dots\cup J_{i_{l}}^{(k)}\]
so that there is a $1\leq p_{l}\leq k$ with $J_{i_{l}}^{(p_{l})}\notin \Lambda(E,\eta,\sigma_{i_{l}}).$ Passing to a further subsequence, we may assume that $p_{l}=p$ is constant. Thus,
\[(J^{(p)}_{i_{l}})_{l\geq 1}\notin \Lambda_{(\sigma_{i_{l}})_{l}},\]
contradicting Proposition \ref{P:automaticLR}.

\end{proof}

\section{A Generalization of Deninger's Problem For Sofic Groups}

Let $\Gamma$ be a countable discrete group. An \emph{algebraic action} of $\Gamma$ is an action $\Gamma\actson X$ by automorphisms, where $X$ is a compact, metrizable, abelian group. An equivalent way to describe this family of actions is to start with a countable $\ZZ(\Gamma)$ module $A,$ and let $\widehat{A}=\Hom(A,\TT)$ where $\TT=\RR/\ZZ$ and $\widehat{A}$ is given the topology of pointwise convergence. We then have the algebraic action $\Gamma\actson \widehat{A}$ by
\[(g\chi)(a)=\chi(g^{-1}a).\]
By Pontryagin duality, all algebraic actions arise in this manner. We will mainly be interested in the case $A=\ZZ(\Gamma)^{\oplus n}/r(f)(\ZZ(\Gamma)^{\oplus m}),$ where $f\in M_{m,n}(\ZZ(\Gamma)),$ in this case $\widehat{A}$ is denoted $X_{f}.$  An interesting aspect of the subject, which has seen great mileage in recent years, (see e.g. \cite{Schmidt} Lemma 1.2, Theorem 1.6, \cite{ChungLi} Theorem 3.1, \cite{Den},\cite{DenSchmidt},\cite{LindSchmidt2},\cite{LiThom}, \cite{LiLiang},\cite{Me4},\cite{Me5}) is that dynamical properties of algebraic actions (i.e. those which only depend upon $\Gamma\actson \widehat{A}$ as an action on a compact metrizable space or probability space) such as entropy and independence tuples of $\Gamma\actson \widehat{A}$ are related to functional analytic objects associated to $\Gamma.$ One relevant object is the group  von Neumann algebra.

The group von Neumann algebra $L(\Gamma)$ is defined to be $\overline{\lambda(\CC(\Gamma))}^{\textnormal{WOT}},$ where $\textnormal{WOT}$ is the weak-operator topology.  Define $\tau\colon L(\Gamma)\to \CC$ by $\tau(x)=\ip{x\delta_{e},\delta_{e}}.$ For $A\in M_{n}(L(\Gamma))$ define
\[\Tr\otimes \tau(A)=\sum_{j=1}^{n}\tau(A_{jj}).\]
Since $L(\Gamma)\subseteq B(\ell^{2}(\Gamma)),$ we can identify $M_{m,n}(L(\Gamma))\subseteq B(\ell^{2}(\Gamma)^{\oplus n},\ell^{2}(\Gamma)^{\oplus m})$ in the natural way. For $x\in M_{m,n}(L(\Gamma)),$ we use $\|x\|_{\infty}$ for the operator norm of $x$ (as an operator $\ell^{2}(\Gamma)^{\oplus n}\to \ell^{2}(\Gamma)^{\oplus m}$). We also use
\[\|x\|_{2}^{2}=\Tr\otimes \tau(x^{*}x).\]
Since we identify $\CC(\Gamma)\subseteq L(\Gamma),$ we will us the same notation for elements of $M_{m,n}(\CC(\Gamma)).$ We shall also identify $\CC(\Gamma)^{\oplus n}\cong M_{1,n}(\CC(\Gamma))$ and use the same notation. For $f\in GL_{n}(L(\Gamma)),$ the Fuglede-Kadison determinant is defined by $\exp\Tr\otimes \tau(\log|f|)$ (here the notation is as in \cite{Me5}). A particular case of Theorem 4.4 in \cite{Me5} shows that if $\Gamma$ is sofic, then
 	 \[h_{(\sigma_{i})_{i}}(X_{f},\Gamma)=\log \Det_{L(\Gamma)}(f),\textnormal{ for $f\in GL_{n}(L(\Gamma))$}\]
(in fact this is true when $f$ is injective as an operator on $\ell^{2}(\Gamma)^{\oplus n}$). When $\Gamma$ is sofic, it is known by \cite{ElekLip} that  for $f\in M_{n}(\ZZ(\Gamma))$ we have $\det_{L(\Gamma)}(f)\geq 1.$ By multiplicativity of Fuglede-Kadison determinants (see \cite{Luck} Theorem 3.14 (1))  it follows that if $f\in GL_{n}(\ZZ(\Gamma)),$ then $\Det_{L(\Gamma)}(f)=1.$ In \cite{DenQuest}, (see question 26) Deninger asked a partial converse to this result. Namely, if $f\in M_{n}(\ZZ(\Gamma))$ is invertible in $GL_{n}(\ell^{1}(\Gamma))$ but not invertible in $M_{n}(\ZZ(\Gamma))$ is $\det_{L(\Gamma)}(f)>1?$

  From Theorem 4.4 of \cite{Me5}, as well as Theorem 6.7 and Proposition 4.16 (3) in \cite{KerrLi2} we can automatically answer Deninger's problem affirmatively for sofic groups. Thus, we automatically have the following.
 \begin{theorem}\label{T:itseasy} Let $\Gamma$ be a countable discrete sofic group and $f\in M_{n}(\ZZ(\Gamma))\cap GL_{n}(\ell^{1}(\Gamma)).$ If $f$ is not in $GL_{n}(\ZZ(\Gamma)),$ then
 \[\Det_{L(\Gamma)}(f)>1.\]
 \end{theorem}
In this section, we show how one can use $\Lambda_{(\sigma_{i})}$-IE-tuples to generalize Deninger's conjecture in the case of sofic groups. In particular, in this section we show the following.

\begin{theorem}\label{T:mainsection4} Let $\Gamma$ be a countable discrete sofic group. If $f\in M_{n}(\ZZ(\Gamma))\cap GL_{n}(L(\Gamma)),$ but is not in $GL_{n}(\ZZ(\Gamma)),$ then $\Det_{L(\Gamma)}(f)>1.$ \end{theorem}

	To illustrate the significance of our generalization, we should mention examples of elements in $M_{n}(\ZZ(\Gamma))$ which are in $GL_{n}(L(\Gamma))$ but are not in $GL_{n}(\ell^{1}(\Gamma)).$ Let $E\subseteq\Gamma,$ and let 
	\[\Delta_{E}=1-\frac{1}{|E|}\sum_{g\in E}g\in \QQ(\Gamma).\]
Note that $\Delta_{E}$ is never invertible in $\ell^{1}(\Gamma).$ To see this, consider the homomorphism
 \[t\colon \ell^{1}(\Gamma)\to \CC\]
 given by
 \[t(f)=\sum_{g\in\Gamma}f(g).\]
 Since $t(\Delta_{E})=0,$ we know that $\Delta_{E}$ is not invertible in $\ell^{1}(\Gamma).$ 
	
	First suppose that $\Gamma$ is a nonamenable group. Let $E\subseteq\Gamma$ be finite and symmetric, i.e. $E=E^{-1}$.
 By nonamenability of $\Gamma,$ we may choose $E$ so that 
\[\frac{1}{|E|}\sum_{g\in E}\lambda(g)\leq 1-\varepsilon\]
for some $\varepsilon>0$ (see e.g. \cite{BO} Theorem 2.6.8 (8)). Thus $\lambda(\Delta_{E})\geq \varepsilon$ as an operator on $\ell^{2}(\Gamma)$ and thus is invertible. So $|E|\Delta_{E}\in\ZZ(\Gamma)\cap L(\Gamma)^{\times}$ but is not in $\ell^{1}(\Gamma)^{\times}$ and thus we always have examples of elements in $\ZZ(\Gamma)\cap L(\Gamma)^{\times}$ which are not invertible in $\ell^{1}(\Gamma)$ if $\Gamma$ is nonamenable. So Theorem \ref{T:mainsection4} applies to these elements whereas Theorem \ref{T:itseasy} does not.

	Theorem \ref{T:mainsection4} is also new in the amenable case. Thus we wish to mention examples when $\Gamma$ is amenable of elements $f\in\ZZ(\Gamma)\cap L(\Gamma)^{\times}$ which are not in $\ell^{1}(\Gamma)^{\times}.$ We say that $\Gamma$ has subexponential growth if for any finite $E\subseteq \Gamma$ we have that 
 \[|\{g_{1}\dots g_{n}:g_{1},\dots,g_{n}\in E\}|^{1/n}\to_{n\to\infty}1.\]
 If $\Gamma$ has subexponential growth then $\alpha\in\CC(\Gamma)$ is invertible in $L(\Gamma)$ if and only if it is invertible in $\ell^{1}(\Gamma)$ by a result of Nica (see \cite{NicaIn}, page 3309). Recall that a group is \emph{virtually nilpotent} if it has a finite index subgroup which is nilpotent. Every virtually nilpotent group has polynomial, and hence subexponential, growth. So if $\Gamma$ is virtually nilpotent then $\alpha\in\CC(\Gamma)$ is invertible in $L(\Gamma)$ if and only if it is invertible in $\ell^{1}(\Gamma).$ The situation is very different when $\Gamma$ does not have subexponential growth. For example, if $\Gamma$ contains a free subsemigroup on two letters, then there are elements $\alpha\in\ZZ(\Gamma)$ which are invertible in $L(\Gamma),$ but not in $\ell^{1}(\Gamma).$ For example, if $g,h$ generate a free subsemigroup in $\Gamma,$ then 
\[\pm 3 e-(e+g+g^{2})h\]
is such an element (see Appendix A of \cite{Li2} for a detailed argument). If $\Gamma$ is a finitely-generated, elementary amenable, not virtually nilpotent group, then a result of Chou say that $\Gamma$ contains a nonabelian free subsemigroup (see \cite{Chou}). Additionally Frey in \cite{Frey} showed that if $\Gamma$ is an amenable group which contains a nonamenable subsemigroup, then it contains a nonabelian free group. For a concrete instance of Chou's result consider the group $\RR\rtimes (\RR\setminus\{0\})$ which is $\RR\times (\RR\setminus\{0\})$ as set but with operation
\[(a,b)(c,d)=(a+bc,bd).\]
If $0\leq a\leq 1/2,$ the subsemigroup generated by $(1,a),(1-a)$ is a free nonabelian semigroup.

	For our purposes, it will be important to use $\ell^{2}-\Lambda_{(\sigma_{i})}-\IE-k$-tuples. Following the methods in our proof of Theorem 4.4 of \cite{Me5}, given an inclusion $B\subseteq A$ of $\ZZ(\Gamma)$-modules, we will want a notion of $\Lambda-(\sigma_{i})_{i}$-IE-tuples corresponding to the inclusion $\widehat{A/B}\subseteq \widehat{A}.$ The use of $\Lambda_{(\sigma_{i})_{i}}-$ independence tuples for inclusions will ease extending Theorem \ref{T:itseasy} to the case when $f$ is only invertible in $M_{n}(L(\Gamma)).$
	
	We will need to recall some notation from \cite{Me5}, as the perturbative techniques there will remain to be important here. For $x\in\RR,$ we use
\[|x+\ZZ|=\inf_{l\in\ZZ}|x-l|.\]
Thus $|\theta|$ makes sense for any $\theta\in\RR/\ZZ.$ Let us recall a definition from \cite{Me4}.
\begin{definition}\label{D:smallonkernel}\emph{Let $\Gamma$ be a countable discrete sofic group with sofic approximation $\sigma_{i}\colon\Gamma\to S_{d_{i}}.$ Let $B\subseteq A$ be countable $\ZZ(\Gamma)$-modules, and let $\rho$ be a continuous dynamically generating pseudometric on $\widehat{A}.$ For finite $F\subseteq\Gamma,D\subseteq B$ and $\delta>0$ we let $\Map(\rho|D,F,\delta,\sigma_{i})$ be all $\phi\in \Map(\rho,F,\delta,\sigma_{i})$ so that}
\[\max_{b\in D}\frac{1}{d_{i}}\sum_{j=1}^{d_{i}}|\phi(j)(b)|^{2}<\delta^{2}.\]
\end{definition}
The main point of this definition is that it is shown in Proposition 4.3 of \cite{Me4} that an element of $\phi\in \Map(\rho|D,F,\delta,\sigma_{i})$ is close to a map
\[\widetilde{\phi}\colon\{1,\dots,d_{i}\}\to \widehat{A/B}\]
which is in $\Map(\rho,F',\delta',\sigma_{i})$ with $\delta'\to 0$ and $F'$ increasing to $\Gamma$ as $\delta\to 0,$ and $F$ increases to $\Gamma,$ and $D$ increases to $B.$ So $\Map(\rho\big|_{\widehat{A/B}},\dots)$ and $\Map(\rho|D,\dots)$ are asymptotically the same notion. A crucial defect of the argument in \cite{Me4}  is that the proof of existence of $\widetilde{\phi}$ is nonconstructive, using a compactness argument in an essential way. However, due to its nonconstructive nature it allows one to create more elements in $\Map(\rho,F,\delta,\sigma_{i})$ than one would initially believe exist. This will be precisely the use here.

We need a similar perturbative idea specifically related to the case of $X_{f}$ for $f\in M_{m,n}(\ZZ(\Gamma)).$ Fix a countable discrete sofic group  $\Gamma$ with sofic approximation $\sigma_{i}\colon\Gamma\to S_{d_{i}}.$ For $x\in \ell^{2}_{\RR}(d_{i},u_{d_{i}})^{\oplus n},$ define
\[\|x\|_{2,(\ZZ^{d_{i}})^{\oplus n}}=\inf_{l\in(\ZZ^{d_{i}})^{\oplus n}}\left(\sum_{j=1}^{n}\|x(j)-l(j)\|_{\ell^{2}(d_{i},u_{d_{i}})}^{2}\right)^{1/2}.\]
For $f\in M_{m,n}(\ZZ(\Gamma)),$ we let
\[\Xi_{\delta}(\sigma_{i}(f))=\{\xi\in (\RR^{d_{i}})^{\oplus n}:\|\sigma_{i}(f)\xi\|_{2,(\ZZ^{d_{i}})^{\oplus m}}<\delta\}.\]

\begin{definition}\emph{ Let $\Gamma$ be a countable discrete sofic group with sofic approximation $\sigma_{i}\colon\Gamma\to S_{d_{i}}.$ Let $B\subseteq A$ be countable $\ZZ(\Gamma)$-modules. Let $\rho$ be a continuous dynamically generating pseudometric for $\Gamma\actson \widehat{A},$ and $1\leq p\leq \infty.$ Fix $x\in\widehat{A/B}^{k}$ and $1\leq q\leq\infty.$ For finite $K,F\subseteq\Gamma,D\subseteq B,E\subseteq\CC(\Gamma)$ and $\eta,\delta>0$ we say that a sequence $J_{i}\subseteq\{1,\dots,d_{i}\}$ is a} $\ell^{q}-\Lambda-(x,\rho|D,F,\delta,E,\eta,(\sigma_{i})_{i};\varepsilon,K)$-independence sequence \emph{if $(J_{i})_{i\geq 1}\in \Lambda_{(\sigma_{i})}$ and for all $c\colon J_{i}\to \{1,\dots,k\}$ so that $c^{-1}(\{l\})\in \Lambda(E,\eta,\sigma_{i})$ there is a $\phi\in\Map(\rho|D,F,\delta,\sigma_{i})$ so that}
\[\max_{g\in K}\rho_{q,J_{i}}(g\phi(\cdot),gx_{c(\cdot)})<\varepsilon.\]
\emph{ We let $I_{\Lambda,q}(x,\rho|D,F,\delta,E,\eta,(\sigma_{i});\varepsilon,K,B\subseteq A)$ be the supremum of}
\[\limsup_{i\to\infty}u_{d_{i}}(J_{i})\]
\emph{ over all $\ell^{q}-\Lambda-(x,\rho|D,F,\delta,E,\eta,(\sigma_{i})_{i};\varepsilon,K)$-independence sequences $(J_{i})_{i\geq 1}.$ Set}
\[I_{\Lambda,q}(x,\rho|D,F,\delta,(\sigma_{i})_{i};\varepsilon,K,B\subseteq A)=\sup_{\substack{\textnormal{ finite} E\subseteq\CC(\Gamma),\\ \eta>0}}I_{\Lambda,q}(x,\rho|D,F,\delta,E,\eta,(\sigma_{i})_{i};\varepsilon,K),\]
\[I_{\Lambda,q}(x,\rho,(\sigma_{i})_{i};\varepsilon,K,B\subseteq A)=\inf_{\substack{\textnormal{finite} D\subseteq B,\\ \textnormal{ finite} F\subseteq\Gamma,\\ \delta>0}}I_{\Lambda,q}(x,\rho|D,F,\delta(\sigma_{i});\varepsilon,K,B\subseteq A).\]
\emph{We say that $x$ is a} $\ell^{q}-\Lambda_{(\sigma_{i})}-\IE-k-$tuple for $B\subseteq A$ \emph{ if for all $\varepsilon>0$ and for all $K\subseteq\Gamma$ finite we have}
\[I_{q}(x,\rho,(\sigma_{i})_{i};\varepsilon,K,B\subseteq A)>0.\]
\emph{We let $\IE^{\Lambda,k}_{(\sigma_{i})_{i}}(\rho,q,B\subseteq A)$ be the set of all $\ell^{q}-\Lambda_{(\sigma_{i})_{i}}-\IE-k$-tuples for $B\subseteq A$.}
\end{definition}

\begin{definition}\emph{ Let $\Gamma$ be a countable discrete sofic group with sofic approximation $\sigma_{i}\colon\Gamma\to S_{d_{i}},$ let $f\in M_{m,n}(\ZZ(\Gamma)).$ Fix $x\in X_{f}^{k}$ and $1\leq q\leq \infty.$ For finite $K\subseteq\Gamma,E\subseteq\CC(\Gamma),$ and $\delta,\eta,\vaerpsilon>0$ we say that a sequence $J_{i}\subseteq\{1,\dots,d_{i}\}$ is a} $\ell^{q}-\Lambda_{(\sigma_{i})}-(x,\delta,E,\eta,(\sigma_{i});\varespilon,K)$ independence sequence for $f$  \emph{if $(J_{i})_{i\geq 1}\in \Lambda_{(\sigma_{i})_{i}}$ and for all $c\colon J_{i}\to \{1,\dots,k\}$ with $c^{-1}(\{l\})\in \Lambda(E,\eta,\sigma_{i})$ there is a $\xi\in \Xi_{\delta}(\sigma_{i}(f))$ so that}
\[\max_{g\in K}\frac{1}{|J_{i}|}\sum_{j\in J_{i}}\sum_{l=1}^{n}|\xi(\sigma_{i}(g)^{-1}(j))(l)-x_{c(j)}(l)(g)|^{2}<\varepsilon^{2}.\]
\emph{We let $I_{\Lambda,q}^{f}(x,\delta,E,\eta,(\sigma_{i});\varepsilon,K)$ be the supremum of}
\[\limsup_{i\to \infty}u_{d_{i}}(J_{i}),\]
\emph{where $J_{i}$ is a $\ell^{q}-\Lambda_{(\sigma_{i})_{i}}-(x,\delta,E,\eta,(\sigma_{i});\varepsilon,K)$ independence sequence. We set}
\[I_{\Lambda,q}^{f}(x,\delta,(\sigma_{i})_{i};\varepsilon,K)=\sup_{\substack{\textnormal{finite} E\subseteq \CC(\Gamma),\\ \eta>0}}I_{\Lambda,q}^{f}(x,\delta,E,\eta,(\sigma_{i})_{i};\varespilon,K)\]
\[I_{\Lambda,q}^{f}(x,(\sigma_{i})_{i};\varepsilon,K)=\inf_{\delta>0}I_{\Lambda,q}^{f}(x,\delta,(\sigma_{i})_{i};\varepsilon,K).\]
\emph{We say that $x$ is a} $\ell^{q}-\Lambda_{(\sigma_{i})_{i}}-\IE-k$-tuple for $f$ \emph{ if for all $\varepsilon>0$ and $K\subseteq\Gamma$ finite we have}
\[I_{\Lambda,q}^{f}(x,(\sigma_{i})_{i};\varepsilon,K)>0.\]
\emph{We use $\IE_{(\sigma_{i})}^{\Lambda,k}(q,f)$ for the set of all $\ell^{q}-\Lambda_{(\sigma_{i})_{i}}-\IE-k$-tuples for $f.$}
\end{definition}

\begin{proposition}\label{P:IEreduction} Let $\Gamma$ be a countable discrete sofic group with sofic approximation $\Sigma.$

(a): Let $B\subseteq A$ be countable $\ZZ(\Gamma)$-modules. Then for $1\leq q<\infty,$ the set of $q-\Lambda_{(\sigma_{i})}$-IE-tuples for $B\subseteq A,$ is the same  as the set of $\ell^{q}-\Lambda_{(\sigma_{i})_{i}}$-IE-tuples for $\Gamma\actson\widehat{A/B}.$

(b): If $f\in M_{m,n}(\ZZ(\Gamma))$ then the set of $\Lambda_{(\sigma_{i})_{i}}$-IE-tuples for $\Gamma\actson X_{f},$ is the same as the set of $\Lambda_{(\sigma_{i})_{i}}$-IE-tuples for $f.$

\end{proposition}

\begin{proof}

(a): Fix $k\in \NN,$ and a continuous dynamically generating pseudometric $\rho$ on $\widehat{A}.$   Use the pseudometric $\rho\big|_{\widehat{A/B}\times \widehat{A/B}}$ on $\widehat{A/B}.$ It is clear that
\[\IE^{\Lambda,k}_{(\sigma_{i})}(\rho,q,B\subseteq A)\supseteq \IE^{\Lambda,k}_{(\sigma_{i})}(q,\widehat{A/B},\Gamma).\]

	For the reserve inclusion let
\[x\in \IE^{\Lambda,k}_{(\sigma_{i})_{i}}(\rho,q,B\subseteq A).\]
Fix finite $K,F\subseteq\Gamma, E\subseteq\CC(\Gamma)$ and $\delta,\eta,\varepsilon>0.$
Set
\[\alpha=I_{\Lambda,q}(x,\rho,(\sigma_{i})_{i};\varepsilon,K,B\subseteq A).\]
Choose finite $F'\subseteq\Gamma,D'\subseteq B,\delta'>0$ in a manner depending upon $F,\delta,\eta$ to be determined later. Let $(J_{i})_{i\geq 1}$ be a $\ell^{q}-\Lambda_{(\sigma_{i})_{i}}-(x,\rho|D,E,\eta,F,\delta,(\sigma_{i})_{i};\varepsilon,K)$ independence set with
\[\limsup_{i\to \infty}u_{d_{i}}(J_{i})\geq \frac{\alpha}{2}.\]
Suppose we are given
\[c\colon J_{i}\to \{1,\dots,k\}\]
with
\[c^{-1}(\{l\})\in \Lambda(E,\eta,\sigma_{i}).\]
Choose a $\phi\in\Map(\rho|D',F',\delta',\sigma_{i})$ with
\[\max_{g\in K}\rho_{q,J_{i}}(g\phi(\cdot),gx_{c(\cdot)})<\varepsilon.\]
Arguing as in Proposition 4.3 in \cite{Me4}, we may find a $\widetidle{\phi}\colon\{1,\dots,d_{i}\}\to \widehat{A/B}$ so that
\[\max_{g\in K\cup F}\max(\rho_{q,d_{i}}(g\phi(\cdot),g\widetilde{\phi}(\cdot)),\rho_{2,d_{i}}(g\phi(\cdot),g\widetilde{\phi}(\cdot)))\leq\kappa_{q}(D',F',\delta'),\]
with
\[\lim_{(D',F',\delta')}\kappa_{q}(D',F',\delta')=0.\]
Here $(D_{1},F_{1},\delta_{1})\leq (D_{2},F_{2},\delta_{2})$ if $D_{1}\subseteq D_{2},F_{1}\subseteq F_{2}$ and $\delta_{1}\geq \delta_{2}.$  Thus
\[\widetidle{\phi}\in \Map(\rho\big|_{\widehat{A/B}\times\widehat{A/B}},F,\delta'+\kappa_{q}(D',F',\delta'),\sigma_{i})\]
and
\[\max_{g\in K\cup F}\rho_{q,J_{i}}(g\widetilde{\phi}(\cdot),gx_{c(\cdot)})<\varepsilon+u_{d_{i}}(J_{i})^{-1}\kappa_{q}(D',F',\delta').\]
Choose $D',F',\delta'$ so that
\[\kappa_{q}(D',F',\delta')+\delta'<\delta\]
\[\kappa_{q}(D',F',\delta')\frac{1}{\alpha}<\varepsilon.\]
Since $\alpha$ does not depend upon $D',F',\delta'$ this is possible. We then see that
\[\widetilde{\phi}\in \Map(\rho\big|_{\widehat{A/B}\times \widehat{A/B}},F,\delta,\sigma_{i})\]
and since
\[\limsup_{i\to\infty}u_{d_{i}}(J_{i})\geq \alpha/2\]
we have
\[\max_{g\in K}\rho_{q,J_{i}}(g\widetilde{\phi}(\cdot),gx_{c(\cdot)})<5\varepsilon\]
for all large $i.$

(b) View $X_{f}\subseteq(\TT^{\Gamma})^{\oplus n},$ and let $\rho$ be the dynamically generating pseudometric on $(\TT^{\Gamma})^{\oplus n}$ defined by
\[\rho(\theta_{1},\theta_{2})^{2}=\sum_{l=1}^{n}|\theta_{1}(l)(e)-\theta_{2}(l)(e)|^{2},\]
where $|\cdot|$ on $\TT$ is in the sense defined in the remarks preceding Definition $\ref{D:smallonkernel}.$ Given $\zeta\in (\TT^{d_{i}})^{\oplus n}$ we can define
\[\phi_{\zeta}\colon\{1,\dots,d_{i}\}\to (\TT^{\Gamma})^{\oplus n}\]
by
\[\phi_{\zeta}(j)(l)(g)=\zeta(\sigma_{i}(g)^{-1}(j))(l).\]
Note that for any $\delta'>0$ and finite $F'\subseteq\Gamma$ we have
\[\phi_{\zeta}\in \Map(\rho,F',\delta',\sigma_{i})\]
for all large $i.$ Indeed for any $h\in \Gamma$
\begin{align*}
\rho(h\phi_{\zeta}(\cdot),\phi_{\zeta}\circ \sigma_{i}(h))^{2}&=\frac{1}{d_{i}}\sum_{l=1}^{n}\sum_{j=1}^{d_{i}}|\zeta(\sigma_{i}(h^{-1})^{-1}(j))(l)-\zeta(\sigma_{i}(e)^{-1}\sigma_{i}(h)(j))(l)|^{2}\\
&\leq nu_{d_{i}}(\{j:\sigma_{i}(h^{-1})^{-1}(j)\ne\sigma_{i}(e)^{-1}\sigma_{i}(h)(j)\})\\
&\to 0,
\end{align*}
the last line following as $(\sigma_{i})_{i}$ is a sofic approximation.
Given $D'\subseteq \ZZ(\Gamma)^{\oplus m}f,$ and $\xi\in\Xi_{\delta}(\sigma_{i}(f))$ by the proof of  Proposition 3.6 in \cite{Me5}, we have that
\[\max_{b\in D}\frac{1}{d_{i}}\sum_{j=1}^{d_{i}}|\ip{\phi_{\xi+(\ZZ^{d_{i}})^{\oplus n}}(j),b}|^{2}<\kappa(\delta)\]
with
\[\lim_{\delta\to 0}\kappa(\delta)=0.\]
Thus $\phi_{\xi+(\ZZ^{d_{i}})^{\oplus n}}\in \Map(\rho|D',F',\delta',\sigma_{i})$ if $\delta$ is sufficiently small and $i$ is sufficiently large. From this it is not hard to argue as in (a) that
\[\IE^{\Lambda,k}_{(\sigma_{i})}(q,f)\subseteq \IE^{\Lambda,k}_{(\sigma_{i})_{i}}(q,\ZZ(\Gamma)^{\oplus m}f\subseteq \ZZ(\Gamma)^{\oplus n}).\]

	Conversely, suppose we have a finite $F'\subseteq\Gamma,$ a $\delta'>0,$ and a finite $D'\subseteq\Gamma.$ Given $\phi\in\Map(\rho|D',F',\delta',\sigma_{i})$ we may define
\[\zeta_{\phi}\in(\TT^{d_{i}})^{\oplus n}\]
by
\[\zeta_{\phi}(l)(j)=\phi(j)(l)(e).\]
Let $\xi_{\phi}\in (\RR^{d_{i}})^{\oplus n}$ be any element such that
\[\xi_{\phi}+(\ZZ^{d_{i}})^{\oplus n}=\zeta_{\phi}.\]
Then by the proof of Proposition 3.6 in \cite{Me5},
\[\xi_{\phi}\in\Xi_{\kappa(D',F',\delta')}(\sigma_{i}(f))\]
with
\[\lim_{(D',F',\delta')}\kappa(D',F',\delta')=0.\]
Here the triples $(D',F',\delta')$ are ordered as in part (a). Again we can use this to argue as in (a) that
\[\IE^{\Lambda,k}_{(\sigma_{i})_{i}}(q,\ZZ(\Gamma)^{\oplus m}f\subseteq\ZZ(\Gamma)^{\oplus n})\subseteq \IE^{\Lambda,k}_{(\sigma_{i})_{i}}(q,f).\]

\end{proof}

We will use the above Lemma to show that
\[\IE^{\Lambda,k}_{(\sigma_{i})_{i}}(2,X_{f},\Gamma)=X_{f}^{k},\]
when $f\in M_{n}(\ZZ(\Gamma))\cap GL_{n}(L(\Gamma)).$ By Proposition \ref{P:regularindependence}, this will prove that every $k$-tuples of points in $X_{f}$ is a $(\sigma_{i})_{i}$-IE-tuple. We first need a way of constructing  elements of $\Lambda_{(\sigma_{i})}$ whose translates by a given finite subset of $\Gamma$ are disjoint. For this we use the following Lemma.

\begin{lemma}\label{L:Bernoulliembedding} Let $\Gamma$ be a countable discrete sofic group with sofic approximation $\sigma_{i}\colon\Gamma\to S_{d_{i}}.$ Fix a finite symmetric subset $E\subset\Gamma$ containing the identity. Then, there is a sequence $(J_{i})_{i\geq 1}\in \Lambda_{(\sigma_{i})_{i}}$ so that
\[\sigma_{i}(x)J_{i}\cap J_{i}=\varnothing\mbox{ for all $x\in E\setminus\{e\}$}\]
\[\lim_{i\to \infty}u_{d_{i}}(J_{i})=\left(\frac{1}{|E|}\right)^{|E|}.\]
\end{lemma}

\begin{proof}

Consider the Bernoulli shift action $\Gamma\acston(E,u_{E})^{\Gamma}.$ Let
\[J=\{x\in E^{\Gamma}:x(g)=g\mbox{ for all $g\in E$}\}.\]
Suppose
\[x\in J\]
and $g\in E\setminus\{e\},$ then
\[(g^{-1}x)(e)=x(g)=g\ne e\]
so $x\notin gJ\cap J.$ Thus $gJ,J$ are disjoint for all $g\in E\setminus\{e\}.$ We now use the fact that Bernoulli shifts have positive sofic entropy to model this behavior on $\{1,\dots,d_{i}\}.$

First note that for every $\varepsilon>0,$ there is a $\delta>0$ so that if $E_{i}\subseteq \{1,\dots,d_{i}\}$ has
\[u_{d_{i}}(\sigma_{i}(x)E_{i}\cap E_{i})\leq \delta,\mbox{ for all $x\in E\setminus\{e\}$}\]
then there is a $E_{i}'\subseteq E_{i}$ with
\[u_{d_{i}}(E_{i}\setminus E_{i}')\leq \varepsilon\]
and
\[\sigma_{i}(x)E_{i}'\cap E_{i}'=\varnothing,\mbox{ $x\in E\setminus\{e\}$}.\]
Indeed, this is simply proved by setting
\[E_{i}'=\bigcap_{x\in E\setminus\{e\}}E_{i}\setminus \sigma_{i}(x)^{-1}(E_{i}).\]

Using the fact that $h_{(\sigma_{i}),u_{E}^{\otimes\Gamma}}(E^{\Gamma},\Gamma)\geq 0,$ we may find a sequence $A_{i,g}\subseteq\{1,\dots,d_{i}\},g\in E$ so that
\begin{equation}\label{E:asymptoticmodelingbehavior}
u_{d_{i}}(\sigma_{i}(g_{1})^{\alpha_{1}}A_{i,h_{1}}\cap \dots \cap \sigma_{i}(g_{k})^{\alpha_{k}}A_{i,h_{k}})\to u_{E}^{\otimes \Gamma}\left(\bigcap_{l=1}^{k}\{x\in E^{\Gamma}:x(g_{l}^{-\alpha_{l}})=h_{k}\}\right)
\end{equation}
for all $k\in \NN,h_{1},\dots,h_{k}\in E,g_{1},\dots,g_{k}\in\Gamma$ and $\alpha_{1},\dots,\alpha_{k}\in \{1,-1\}$ (see e.g. Bowen's original definition of sofic entropy in \cite{Bow}). Set
\[J_{i}'=\bigcap_{g\in E}\sigma_{i}(g)^{-1}A_{i,g},\]
then by (\ref{E:asymptoticmodelingbehavior}),
\[u_{d_{i}}(\sigma_{i}(g)J_{i}'\cap J_{i}')\to 0\]
for all $g\in E\setminus\{e\}$ and
\[u_{d_{i}}(\sigma_{i}(x)J_{i}'\cap J_{i}')\to u_{E}^{\otimes \Gamma}(xJ\cap J)\]
for all $x\in\Gamma.$ Applying our previous observation we find $J_{i}\subseteq J_{i}'$ so that
\[u_{d_{i}}(J_{i}'\setminus J_{i})\to 0\]
and
\[\sigma_{i}(g)J_{i}\cap J_{i}=\varnothing\mbox{ for $g\in E\setminus\{e\}$}.\]
Since $u_{d_{i}}(J_{i}'\setminus J_{i})\to 0,$ we have
\begin{equation}\label{E:matrixcoeff}
u_{d_{i}}(\sigma_{i}(g)J_{i}\cap J_{i})-u_{d_{i}}(J_{i})^{2}\to u_{E}^{\otimes \Gamma}(gJ\cap J)-u_{E}^{\otimes \Gamma}(J)^{2}
\end{equation}
for all $g\in\Gamma.$ It is well-known that $\Gamma\actson (L^{2}((E,u_{E})^{\Gamma})\ominus \CC1)$ can be equivariantly, isometrically embedded in $\Gamma\actson \ell^{2}(\NN\times\Gamma)$ where the action of $\Gamma\actson \ell^{2}(\NN\times\Gamma)$ is given by
\[(g\xi)(n,h)=\xi(n,g^{-1}h).\]
 We will use this to show that $(J_{i})_{i\geq 1}\in \Lambda_{(\sigma_{i})_{i}}.$  We again use $o(1)$ for any expression which goes to zero as $i\to\infty.$ Let $\alpha\colon\Gamma\to \mathcal{U}(L^{2}((E,u_{E})^{\Gamma})\ominus \CC1)$ be the representation 
 \[(\alpha(g)\xi)(\omega)=\xi(g^{-1}\omega), \mbox{ $\omega\in E^{\Gamma},g\in\Gamma$}.\]
 Extend by linearity to a $*$-representation 
 \[\alpha\colon\CC(\Gamma)\to B(L^{2}(E,u_{E})^{\Gamma}\ominus \CC1).\]
 Then for any $f\in \CC(\Gamma)$ and $i\in \NN$ we have
 \begin{align}\label{E:sec4alskj}
 \|\sigma_{i}(f)(\chi_{J_{i}}-u_{d_{i}}(J_{i})1)\|_{2}^{2}&=\sum_{g,h\in\Gamma}\widehat{f}(g)\overline{\widehat{f}(h)}\ip{\sigma_{i}(g)(\chi_{J_{i}}-u_{d_{i}}(J_{i})1),\sigma_{i}(h)(\chi_{J_{i}}-u_{d_{i}}(J_{i})1)}\\ \nonumber
 &=o(1)+\sum_{g,h\in\Gamma}\widehat{f}(g)\overline{\widehat{f}(h)}\ip{\sigma_{i}(h^{-1}g)(\chi_{J_{i}}-u_{d_{i}}(J_{i})1),\chi_{J_{i}}-u_{d_{i}}(J_{i})1}\\ \nonumber
&=o(1)+ \sum_{g\in\Gamma}\widehat{f^{*}f}(g)\ip{\sigma_{i}(g)(\chi_{J_{i}}-u_{d_{i}}(J_{i})1),\chi_{J_{i}}-u_{d_{i}}(J_{i})1}\\ \nonumber
 &=o(1)+\sum_{g\in\Gamma}\widehat{f^{*}f}(g)(u_{d_{i}}(\sigma_{i}(g)J_{i}\cap J_{i})-u_{d_{i}}(J_{i})^{2})\\ \nonumber
 &=o(1)+\sum_{g\in\Gamma}\widehat{f^{*}f}(g)(u_{E}^{\otimes \Gamma}(gJ\cap J)-u_{E}^{\otimes \Gamma}(J)^{2})\\ \nonumber
 &=o(1)+\sum_{g\in\Gamma}\widehat{f^{*}f}(g)\ip{\alpha(g)(\chi_{J}-u_{E}^{\otimes \Gamma}(J)1),\chi_{E}-u_{E}^{\otimes \Gamma}(J)1}\\ \nonumber
 &=o(1)+\|\alpha(f)(\chi_{J}-u_{E}^{\otimes \Gamma}(J)1)\|_{2}^{2}.
 \end{align}
 Since $\alpha$ can be embedded into the infinite direct sum of the left regular representation, we have that 
 \begin{align}\label{E:sec4alsf}
\|\alpha(f)(\chi_{J}-u_{E}^{\otimes \Gamma}(J)1)\|_{2}&\leq \|\lambda(f)\|\|\chi_{J}-u_{E}^{\otimes \Gamma}(J)1\|_{2}\\ \nonumber
 &= \|\lambda(f)\|(u_{E}^{\otimes \Gamma}(J)-u_{E}^{\otimes\Gamma}(J)^{2})^{1/2}\\ \nonumber
 &=o(1)+\|\lambda(f)\|(u_{d_{i}}(J_{i})-u_{d_{i}}(J_{i})^{2})^{1/2}\\ \nonumber
 &=o(1)+\|\lambda(f)\|\|\chi_{J_{i}}-u_{d_{i}}(J_{i})1\|_{2}.
 \end{align}
 By $(\ref{E:sec4alskj}),(\ref{E:sec4alsf})$ we have that $(J_{i})_{i}\in \Lambda_{(\sigma_{i})_{i}}.$
 From our construction it also follows that
\[u_{d_{i}}(J_{i})\to u_{E}^{\otimes \Gamma}(J)=\left(\frac{1}{|E|}\right)^{|E|}.\]

\end{proof}

For the next Lemma, we need some notation. We use $\TT=\RR/\ZZ.$  Define $t\colon\CC(\Gamma)\to\CC$ by
\[t(\alpha)=\sum_{g\in\Gamma}\widehat{\alpha}(g).\]
\begin{lemma} Let $\Gamma$ be a countable discrete group, and $f\in M_{n}(\ZZ(\Gamma))\cap GL_{n}(L(\Gamma)).$ Let $\phi$ be the inverse of $f$ in $M_{n}(L(\Gamma)).$  Define $Q\colon \ell^{2}_{\RR}(\Gamma)^{\oplus n}\to (\TT^{\Gamma})^{\oplus n}$ by
\[Q(\xi)(l)(g)=\xi(l)(g)+\ZZ.\]
Then $Q(\{\alpha\phi^{*}:\alpha\in\ZZ(\Gamma)^{\oplus n},t(\alpha(j))=0,1\leq j\leq n\})$ is dense in $X_{f}.$

\end{lemma}

\begin{proof}
As usual, we view $\ZZ(\Gamma)^{\oplus n}\subseteq \ell^{2}(\Gamma)^{\oplus n}.$ For $\alpha,\beta\in\ZZ(\Gamma)^{\oplus n}$ we have 
\[\ip{\alpha,\beta}=\sum_{l=1}^{n}\tau(\beta(l)^{*}\alpha(l))\]
where $\tau$ is the trace on $L(\Gamma).$  For $\theta\in(\TT^{\Gamma})^{\oplus n},\alpha\in\ZZ(\Gamma)^{\oplus n}$ we set
\[\ip{\theta,\alpha}_{\TT}=\sum_{l=1}^{n}\sum_{g\in\Gamma}\theta(l)(g)\widehat{\alpha(l)}(g)\in\TT.\]
Then the pairing $\ip{\cdot,\cdot}_{\TT}$ allows us to identify $(\TT^{\Gamma})^{\oplus n}\cong (\ZZ(\Gamma)^{\oplus n})^{\widehat{}}.$

By Pontryagin duality, it suffices to show that if $\beta\in\ZZ(\Gamma)^{\oplus n}$ has
\[\ip{\beta,\alpha\phi^{*}}\in\ZZ\]
for all $\alpha\in\ZZ(\Gamma)^{\oplus n}$ with $t(\alpha(l))=0,1\leq l\leq n,$ then $\beta\in \ZZ(\Gamma)^{\oplus n}f.$ For $x\in L(\Gamma),$ and $1\leq l\leq n$ we use $x\otimes e_{l}\in L(\Gamma)^{\oplus n}$ which is $x$ in the $l^{th}$ coordinate and $0$ in every other coordinate. Fix $1\leq l\leq n$ and consider $\alpha=(g-1)\otimes e_{l}.$ Then
\[\ip{\beta,\alpha\phi^{*}}=\ip{\beta\phi,\alpha}=\widehat{(\beta\phi)(l)}(g)-\widehat{(\beta\phi)(l)}(e).\]
So
\[\widehat{(\beta\phi)(l)}(g)-\widehat{(\beta\phi)(l)}(e)\in\ZZ\]
for all $g\in\Gamma.$ Letting $g\to\infty,$ and using that $\widehat{(\beta\phi)(l)}\in\ell^{2}$ we find that
\[\widehat{(\beta\phi)(l)}(e)\in\ZZ.\]
As
\[\widehat{(\beta\phi)(l)}(g)-\widehat{(\beta\phi)(l)}(e)\in\ZZ\]
for all $g\in\Gamma,1\leq l\leq n$ we find that $\beta\phi\in\ZZ(\Gamma)^{\oplus n}.$ Thus
\[\beta=(\beta\phi)f\in\ZZ(\Gamma)^{\oplus n}f.\]

\end{proof}

We are now ready to prove our Theorem, but we first recall the notation we introduced at the beginning of this section. From the identifications 
\[M_{m,n}(\CC(\Gamma))\subseteq M_{m,n}(L(\Gamma))\subseteq B(\ell^{2}(\Gamma)^{\oplus n},\ell^{2}(\Gamma)^{\oplus m})\]
we may think of elements of $M_{m,n}(\CC(\Gamma))$ as bounded, linear operators $\ell^{2}(\Gamma)^{\oplus n}\to \ell^{2}(\Gamma)^{\oplus m}.$
For a fixed $x\in M_{m,n}(\CC(\Gamma))$ we let $\|x\|_{\infty}$ be the norm of $x$ as an operator  $\ell^{2}(\Gamma)^{\oplus n}\to \ell^{2}(\Gamma)^{\oplus m}$
under the above identification. We also identify $\CC(\Gamma)^{\oplus n}\cong M_{1,n}(\CC(\Gamma))$ and use the notation above. We thus caution the reader that for $A\in M_{m,n}(\CC(\Gamma))$ 
\[\|A\|_{\infty}\ne \sup_{\substack{g\in\Gamma,\\ 1\leq i\leq m,1\leq j\leq n}}|\widehat{A_{ij}}(g)|,\]
with similar remarks for elements of $\CC(\Gamma)^{\oplus n}.$ 

\begin{theorem}\label{T:tuplesmanf} Let $\Gamma$ be a countable discrete group with sofic approximation $\sigma_{i}\colon \Gamma\to S_{d_{i}}$ be a sofic approximation. Let $f\in M_{n}(\ZZ(\Gamma))\cap GL_{n}(L(\Gamma)),$ then every $k$-tuple of points in $X_{f}$ is a $\ell^{2}-\Lambda_{(\sigma_{i})_{i}}-\IE-k$-tuple.

\end{theorem}

\begin{proof}

   Let $\phi$ be the inverse of $f$ in $M_{n}(L(\Gamma)).$  By the preceding lemma and Proposition \ref{P:permanence}, it suffices to prove the theorem when $(x_{1},\dots,x_{k})=(Q(\alpha_{1}\phi^{*}),\dots,Q(\alpha_{k}\phi^{*}))$ where $t(\alpha_{j})=0.$ For $t>0,$ let  $\phi_{t}\in M_{n}(\RR(\Gamma))$ be such that
\[\|\phi_{t}-\phi\|_{\infty}<t.\]

Fix $\varepsilon>0,$ and a $A\subseteq\Gamma$ finite. Suppose we are given a finite $F\subseteq\Gamma,$ and a $\delta>0.$ Let $E\subseteq\CC(\Gamma)$ be finite and $\eta>0$ to depend upon $F,\delta$ in a manner to be determined later.
Let
\[L_{1}=(\supp(\phi_{\varepsilon})\cup \{e\}\cup \supp(\phi_{\varepsilon})^{-1}),\]
\[L_{2,s}=(\supp(\alpha_{s})\cup \{e\}\cup\supp(\alpha_{s})^{-1})\mbox{ for $1\leq s\leq k,$}\]
\[K_{1}=\left[\bigcup_{1\leq s\leq k}L_{2,s}(\supp(f)\cup\{e\}\cup\supp(f)^{-1})L_{1}\right]^{(2015)!},\]
\[K_{2}=\left[\bigcup_{1\leq s\leq k}((A\cup\{e\}\cup A^{-1})(\supp(\phi_{\varepsilon})\cup\{e\}\cup\supp(\phi_{\varepsilon})^{-1})(\supp(\alpha_{s})\cup\{e\}\cup\supp(\alpha_{s})^{-1}))\right]^{(2015)!},\]
\[K=K_{1}\cup K_{2}.\]
Apply Lemma \ref{L:Bernoulliembedding} to find a sequence $J_{i}\subseteq \{1,\dots,d_{i}\}$ so that $\{\sigma_{i}(x)J_{i}\}_{x\in K}$ are a disjoint family, and $(J_{i})_{i}\in\Lambda_{(\sigma_{i})}$ and
\[\lim_{i\to \infty}\frac{|J_{i}|}{d_{i}}=\left(\frac{1}{|K|}\right)^{|K|}.\]
Note that if $J_{i}'\subseteq J_{i}$ satisfies
\[u_{d_{i}}(J_{i}\setminus J_{i}')\to 0,\]
then $J_{i}'$ enjoys the conclusions of Lemma \ref{L:Bernoulliembedding} as well. So by soficity, we may assume
\[\sigma_{i}(x)(j)\ne \sigma_{i}(y)(j)\]
for $x\ne y\in K,j\in J_{i}$ and that
\[\sigma_{i}(x_{1}\dots x_{l})(j)=\sigma_{i}(x_{1})\dots\sigma_{i}(x_{l})(j)\]
for $x_{1},\dots,x_{l}\in K$ and $1\leq l\leq (2015)!.$ Let $c\colon J_{i}\to \{1,\dots,k\}$ be such that $c^{-1}(\{s\})\in \Lambda(E,\eta,\sigma_{i}).$ Set
\[J_{i}^{(s)}=c^{-1}(\{s\}).\]
 For $t\in (0,\infty)$ let
\[\xi_{t}=\sum_{1\leq s\leq k}\sigma_{i}(\phi_{t}\alpha_{s}^{*})\chi_{J_{i}^{(s)}}.\]
Note that

\[\sigma_{i}(f)\xi_{\delta}-\sum_{1\leq s\leq k}\sigma_{i}(\alpha_{s}^{*})\chi_{J_{i}^{(s)}}=\sum_{1\leq s\leq k}(\sigma_{i}(f)\sigma_{i}(\phi_{\delta}\alpha_{s}^{*})-\sigma_{i}(\alpha_{s}^{*}))\chi_{J_{i}^{(s)}}.\]
For $\beta\in \CC(\Gamma)$ we have
\[\sigma_{i}(\beta)1=t(\beta)1.\]
Because $t(\alpha_{s})=0$ for $1\leq s\leq k,$
\begin{equation}\label{E:cancellation1}
\sigma_{i}(f)\xi_{\delta}-\sum_{1\leq s\leq k}\sigma_{i}(\alpha_{s}^{*})\chi_{J_{i}^{(s)}}=\sum_{1\leq s\leq k}(\sigma_{i}(f)\sigma_{i}(\phi_{\delta}\alpha_{s}^{*})-\sigma_{i}(\alpha_{s}^{*}))(\chi_{J_{i}^{(s)}}-u_{d_{i}}(\chi_{J_{i}^{(s)}})1).
\end{equation}
Since
\[\|\chi_{J_{i}^{(s)}}-u_{d_{i}}(J_{i}^{(s)})1\|_{\infty}\leq 2,\]
we have
\[\|\sigma_{i}(f)\sigma_{i}(\phi_{\delta}\alpha_{s}^{*})(\chi_{J_{i}^{(s)}}-u_{d_{i}}(J_{i}^{(s)})1)-\sigma_{i}(f\phi_{\delta}\alpha_{s}^{*})(\chi_{J_{i}^{(s)}}-u_{d_{i}}(J_{i}^{(s)})1)\|_{2}\to_{i\to\infty}0.\]
Thus
\begin{equation}\label{E:cancellation2}
\left\|\sum_{1\leq s\leq k}\left(\sigma_{i}(f)\sigma_{i}(\phi_{\delta}\alpha_{s}^{*})(\chi_{J_{i}^{(s)}}-u_{d_{i}}(J_{i}^{(s)})1)-\sigma_{i}(f\phi_{\delta}\alpha_{s}^{*}))(\chi_{J_{i}^{(s)}}-u_{d_{i}}(J_{i}^{(s)})1)\right)\right\|_{2}\to_{i\to\infty}0.
 \end{equation}
 We have
 \[\|\sigma_{i}(f\phi_{\delta}\alpha_{s}^{*})(\chi_{J_{i}^{(s)}}-u_{d_{i}}(J_{i}^{(s)})1)-\sigma_{i}(\alpha_{s}^{*})(\chi_{J_{i}^{(s)}}-u_{d_{i}}(J_{i}^{(s)})1)\|_{2}=\]
 \[\left(\sum_{l=1}^{n}\left\|\sum_{p=1}^{n}(\sigma_{i}((f\phi_{\delta}\alpha_{s}^{*}-\alpha_{s}^{*})_{lp})(\chi_{J_{i}^{(s)}}-u_{d_{i}}(J_{i}^{(s)})1)\right\|_{2}^{2}\right)^{1/2}.\]
If
 \[E\supseteq\{(f\phi_{\delta}\alpha_{s}^{*}-\alpha_{s}^{*})_{lp}:1\leq l,p\leq n\}\]
 then as
 \[\|f\phi_{\delta}\alpha_{s}^{*}-\alpha_{s}^{*}\|_{\infty}\leq \delta\|f\|_{\infty}\|\alpha_{s}\|_{\infty},\]
 we have
 \begin{align*}
 \|\sigma_{i}(f\phi_{\delta}\alpha_{s}^{*})(\chi_{J_{i}^{(s)}}-u_{d_{i}}(J_{i}^{(s)})1)-\sigma_{i}(\alpha_{s}^{*})(\chi_{J_{i}^{(s)}}&-u_{d_{i}}(J_{i}^{(s)})1)\|_{2}\\
 &\leq\left(\sum_{l=1}^{n}\left( \|\alpha_{s}\|_{\infty}\|f\|_{\infty}n\delta\|\chi_{J_{i}^{(s)}}-u_{d_{i}}(J_{i}^{(s)}1)\|_{2}+n\eta\right)^{2}\right)^{1/2}\\
&\leq \|\alpha_{s}\|_{\infty}\|f\|_{\infty}n^{2}\delta\|\chi_{J_{i}^{(s)}}\|_{2}+n^{2}\eta.
 \end{align*}
 Set
\[M=(\|f\|_{\infty}+1)\left(\sum_{1\leq s\leq k}\|\alpha_{s}\|_{\infty}^{2}\right)^{1/2},\]
then
 \begin{align}\label{E:cancellation3}
\left\|\sum_{1\leq s\leq k}\sigma_{i}(f\phi_{\delta}\alpha_{s}^{*})(\chi_{J_{i}^{(s)}}-u_{d_{i}}(J_{i}^{(s)})1)-\sigma_{i}(\alpha_{s}^{*})(\chi_{J_{i}^{(s)}}-u_{d_{i}}(J_{i}^{(s)})1)\right\|_{2}&\leq kn^{2}\eta\\ \nonumber
&+n^{2}\delta\|f\|_{\infty}\sum_{1\leq s\leq k}\|\alpha_{s}\|_{\infty}\|\chi_{J_{i}^{(s)}}\|_{2}\\ \nonumber
 &\leq kn^{2}\eta+Mn^{2}\delta u_{d_{i}}(J_{i})^{1/2},\nonumber
\end{align}
where in the last step we use the Cauchy-Schwartz inequality and the fact that
\[\sum_{1\leq s\leq k}\|\chi_{J_{i}^{(s)}}\|_{2}^{2}=u_{d_{i}}(J_{i}).\]
 If we force $\eta$ sufficiently small then by $(\ref{E:cancellation1}),(\ref{E:cancellation2}),(\ref{E:cancellation3})$ we have for all large $i,$
 \begin{equation}\label{E:itsmicrostate}
 \xi_{\delta}\in \Xi_{\delta(n^{2}+1)M}(\sigma_{i}(f)).
 \end{equation}
We will want to force $\eta$ to be even smaller later.

If $E\supseteq \{(\phi_{\varepsilon}\alpha_{s}^{*}-\phi_{\delta}\alpha_{s}^{*})_{pl}:1\leq l,p\leq n\},$ then for all $1\leq s\leq k,$ for all $1\leq p,l\leq n$
\[\|\sigma_{i}((\phi_{\varepsilon}\alpha_{s}^{*}-\phi_{\delta}\alpha_{s}^{*})_{pl})(\chi_{J_{i}^{(s)}}-u_{d_{i}}(J_{i}^{(s)})1)\|_{2}\leq 2\varepsilon\|\alpha_{s}\|_{\infty}\|\chi_{J_{i}^{(s)}}-u_{d_{i}}(J_{i}^{(s)})1)\|_{2}+\eta\leq 2\varepsilon\|\alpha_{s}\|_{\infty}u_{d_{i}}(J_{i}^{(s)})^{1/2}+\eta. \]
Note that in our defintion of $\ell^{2}-\Lambda_{(\sigma_{i})_{i}}$-tuples we are allowed to have $E,\eta$ depend upon $\delta.$ 
By the same arguments as before
\[\|\xi_{\varepsilon}-\xi_{\delta}\|_{2,J_{i}}\leq n^{2}\eta+2\varepsilon n^{2}\sum_{1\leq s\leq k}\|\alpha_{s}\|_{\infty}u_{d_{i}}(J_{i}^{(s)})^{1/2}\\
\leq n^{2}\eta+2\varepsilon M u_{d_{i}}(J_{i})^{1/2}\]
where again we have used the Cauchy-Schwartz inequality and the fact that
\[\sum_{1\leq s\leq k}\|\chi_{J_{i}^{(s)}}\|_{2}^{2}=u_{d_{i}}(J_{i}).\]
Thus
\[\frac{1}{|J_{i}|}\sum_{j\in J_{i}}|\xi_{\varepsilon}(j)-\xi_{\delta}(j)|^{2}=u_{d_{i}}(J_{i})^{-1}\frac{1}{d_{i}}\sum_{j\in J_{i}}|\xi_{\varepsilon}(j)-\xi_{\delta}(j)|^{2}\leq (u_{d_{i}}(J_{i})^{-1}n^{4}\eta^{2}+ 2\varepsilon Mn^{2}\eta u_{d_{i}}(J_{i})^{-1/2}+4\varepsilon^{2} M^{2}).\]
For all large $i,$
\[u_{d_{i}}(J_{i})\geq \frac{1}{2}\left(\frac{1}{|K|}\right)^{|K|}.\]
So we can choose $\eta$ sufficiently small (depending only upon $K$) so that
\[\|\xi_{\delta}-\xi_{\varepsilon}\|_{2,J_{i}}<\varepsilon(2M+1).\]
Then
\begin{equation}\label{E:closenessetc}
\left(\frac{1}{|J_{i}|}\sum_{j\in J_{i}}\left|[\sigma_{i}(g)\xi_{\delta}](j)-[\sigma_{i}(g)\xi_{\varpesilon}](j)+\ZZ\right|^{2}\right)^{1/2}\leq \|\sigma_{i}(g)\xi_{\delta}-\sigma_{i}(g)\xi_{\varepsilon}\|_{2,J_{i}}<\varepsilon(2M+1).
\end{equation}
Additionally, for $j\in J_{i}^{(s)},g\in A$
\begin{align*}
\sigma_{i}(g)\xi_{\varepsilon}(j)&=\sum_{x\in \Gamma}\sum_{1\leq s\leq k}\widehat{\phi_{\varepsilon}\alpha_{s}^{*}}(x)\chi_{\sigma_{i}(g)\sigma_{i}(x)J_{i}^{(s)}}(j)\\
&=\sum_{x\in K\cap g^{-1}K}\sum_{1\leq s\leq k}\widehat{\phi_{\varepsilon}\alpha_{s}^{*}}(x)\chi_{\sigma_{i}(g)\sigma_{i}(x)J_{i}^{(s)}}(j),
\end{align*}
here we use our choice of $J_{i}$ as well as the fact that $K\cap g^{-1}K\supseteq \supp(\phi_{\varepsilon}\alpha_{s}^{*}).$ As $\{\sigma_{i}(k)J_{i}\}_{k\in K}$ are a disjoint family, we have for $x\in K\cap g^{-1}K$ that $\chi_{\sigma_{i}(gx)J_{i}^{(s)}}(j)=1$ if and only if $gx=e,$ and thus when $x=g^{-1}.$  Since $K\cap g^{-1}K\supseteq \supp(\phi_{\varepsilon}\alpha_{s}^{*}),$ the above sum is
\[\widehat{\phi_{\varepsilon}\alpha_{s}^{*}}(g^{-1})=\widehat{\alpha_{s}\phi_{\varepsilon}^{*}}(g).\]
As
\[|\widehat{\alpha_{s}\phi_{\varepsilon}^{*}}(g)-\widehat{\alpha_{s}\phi^{*}}(g)|\leq \|\alpha_{s}\phi_{\varepsilon}^{*}-\alpha_{s}\phi^{*}\|_{2}\leq \varepsilon\|\alpha_{s}\|_{2}\leq \varepsilon\|\alpha_{s}\|_{\infty}.\]
We find that
\[\max_{g\in A}\left(\frac{1}{|J_{i}|}\sum_{j\in J_{i}}|(\sigma_{i}(g)\xi_{\varepsilon})(j)+\ZZ-\widehat{\alpha_{c(j)}\phi^{*}}(g)|^{2}\right)^{1/2}<\varepsilon M.\]
Combining with (\ref{E:closenessetc})
\[\max_{g\in A}\left(\frac{1}{|J_{i}|}\sum_{j\in J_{i}}|(\sigma_{i}(g)\xi_{\delta})(j)+\ZZ-\widehat{\alpha_{c(j)}\phi^{*}}(g)|^{2}\right)^{1/2}<\varepsilon(3M+1).\]
As $\varepsilon>0$ is arbitrary, the Theorem is now proved using Proposition \ref{P:IEreduction}.

\end{proof}

\begin{cor}\label{C:IEandregular} Let $\Gamma$ be a countable discrete sofic group with sofic approximation $\sigma_{i}\colon\Gamma\to S_{d_{i}}.$ Let $f\in M_{n}(\ZZ(\Gamma))\cap GL_{n}(L(\Gamma)).$  Then every $k$-tuple of points in $X_{f}$ is a $(\sigma_{i})_{i}-\IE-k$-tuple.
\end{cor}
\begin{proof}

Automatic from the preceding Theorem and Proposition \ref{P:regularindependence}.

\end{proof}

\begin{cor} Let $\Gamma$ be a countable discrete sofic group and $f\in M_{n}(\ZZ(\Gamma))\cap GL_{n}(L(\Gamma)).$ If $f$ is not in $GL_{n}(\ZZ(\Gamma)),$ then $\det_{L(\Gamma)}(f)>1.$

\end{cor}

\begin{proof}

Observe that $X_{f}$ is not a single point if and only if $f\notin GL_{n}(\ZZ(\Gamma)).$ The corollary is then automatic from the preceding Corollary, Proposition 4.16 (3) in \cite{KerrLi2} and Theorem 4.4 in \cite{Me5}.

\end{proof}

In fact, we have the following more general result. Recall that if $\Gamma$ is sofic, if $\sigma_{i}\colon\Gamma\to S_{d_{i}}$ is a sofic approximation, an action $\Gamma\actson X$ on a compact metrizable space is said to have \emph{completely positive topological entropy} with respect to $(\sigma_{i})_{i}$ if whenever $\Gamma\actson Y$ is a nontrivial (i.e. not a one-point space) topological factor of $X,$ we have $h_{(\sigma_{i})_{i}}(Y,\Gamma)>0.$ The following Corollary was known for $f\in M_{n}(\ZZ(\Gamma))\cap GL_{n}(\ell^{1}(\Gamma)),$ by Proposition 4.16 (3),(5) and Theorem 6.7 of \cite{KerrLi2}. 

\begin{cor}\label{C:topcpe} Let $\Gamma$ be a countable discrete sofic group and $f\in M_{n}(\ZZ(\Gamma))\cap GL_{n}(L(\Gamma)).$ Suppose that $f$ is not in $GL_{n}(\ZZ(\Gamma)).$ Then for any sofic approximation $\sigma_{i}\colon\Gamma\to S_{d_{i}},$ the action $\Gamma\actson X_{f}$ has completely positive topological entropy with respect to $(\sigma_{i})_{i}.$
\end{cor}
	
\begin{proof}
Automatic from Theorem \ref{T:tuplesmanf},  Proposition \ref{P:regularindependence}, Proposition 4.16 (3) in \cite{KerrLi2} and Proposition \ref{P:permanence}.

\end{proof}

Combining with results of Chung-Li we have the following result in the amenable case, which previously only known for $f\in GL_{n}(\ell^{1}(\Gamma))$ (see Corollary 8.4, Theorem 7.8 and Lemma 5.4 of \cite{ChungLi}).

\begin{cor}\label{C:cpeamenable} Let $\Gamma$ be a countable amenable group, and $f\in M_{n}(\ZZ(\Gamma))\cap GL_{n}(L(\Gamma)).$ Suppose that $f$ is not in $GL_{n}(\ZZ(\Gamma)).$ Then the action $\Gamma\actson X_{f}$ has completely positive measure-theoretic entropy (with respect to the Haar measure on $X_{f}$).
\end{cor}

\begin{proof} This follows from Corollary 8.4 of \cite{ChungLi} and Corollary \ref{C:IEandregular}.

\end{proof}
The preceding corollary was known in the amenable case when $f\in M_{n}(\ZZ(\Gamma))\cap GL_{n}(\ell^{1}(\Gamma))$ by combining   Proposition 4.16 (3),(5) and Theorem 6.7 of \cite{KerrLi2} with Corollary 8.4 of \cite{ChungLi}. As we mentioned, at the beginning of this section there are interesting examples in the amenable case of $f\in \ZZ(\Gamma)\cap L(\Gamma)^{\times}$ but $f\notin \ell^{1}(\Gamma)^{\times}.$ When $\Gamma$ is sofic, it would be interesting to decide if $\Gamma\actson X_{f}$ has completely positive measure-theoretic entropy with respect to every sofic approximation if $f\in M_{n}(\ZZ(\Gamma))\cap GL_{n}(L(\Gamma))$ is not invertible in $M_{n}(\ZZ(\Gamma)).$ 

\end{document}